\documentclass[a4paper,oneside,10pt]{article} 

\title{A Self-Interaction Leading to Fluctuations of Order $n^{5/6}$}
\author{Matthias Gorny}
\usepackage[english,frenchb]{babel}
\usepackage[latin1]{inputenc}
\usepackage[T1]{fontenc}
\usepackage{amsfonts}
\usepackage{graphicx}
\usepackage[cyr]{aeguill}
\usepackage{xspace}
\usepackage{tikz}
\usetikzlibrary{patterns}
\usepackage{cite}
\usepackage{tocloft}
\usepackage{amsthm}

\usepackage{amssymb}
\usepackage{amsmath}
\usepackage{stmaryrd}
\usepackage{latexsym}
\usepackage{exscale}
\usepackage{relsize}
\usepackage{accents}
\usepackage{dsfont}
\usepackage{mathrsfs}
\usepackage{yhmath}

\newcommand{\ind}[1]{\mathds{1}_{#1}}
\newcommand{\N}{\mathbb{N}}

\newcommand{\R}{\mathbb{R}}

\newcommand{\E}{\mathbb{E}}

\newcommand{\Ll}{\mathrm{L}}
\newcommand{\Bl}{\mathrm{B}}

\newcommand{\Hc}{\mathcal{H}}

\newcommand{\Lc}{\mathcal{L}}

\newcommand{\Pc}{\mathcal{P}}

\newcommand{\Sc}{\mathcal{S}}
\newcommand{\Tc}{\mathcal{T}}
\newcommand{\Uc}{\mathcal{U}}

\newcommand{\Ck}[1]{\mathcal{C}^{#1}}

\def\Dro{\smash{{D}^{\!\!\!\!\raise4pt\hbox{$\scriptstyle o$}}}}
\def\Ccro{\smash{{\mathcal{C}}^{\!\!\!\raise4pt\hbox{$\scriptstyle o$}}}}
\def\Aro{\smash{{A}^{\!\!\!\raise5pt\hbox{$\scriptstyle o$}}}}
\def\Bro{\smash{{B}^{\!\!\!\raise5pt\hbox{$\scriptstyle o$}}}}

\newcommand{\limsupn}{\underset{n \to +\infty}{\mathrm{limsup}}\,}
\newcommand{\liminfn}{\underset{n \to +\infty}{\mathrm{liminf}}\,}

\renewcommand{\a}{\alpha}
\renewcommand{\b}{\beta}
\newcommand{\g}{\gamma}
\newcommand{\G}{\Gamma}
\renewcommand{\d}{\delta}

\renewcommand{\epsilon}{\varepsilon}
\newcommand{\eps}{\varepsilon}

\renewcommand{\r}{\rho}
\newcommand{\s}{\sigma}

\renewcommand{\phi}{\varphi}

\renewcommand{\L}{\Lambda}

\newtheorem{theo}{Theorem}

\newtheorem{hyps}[theo]{Hypothesis}
\newtheorem{prop}[theo]{Proposition}

\newtheorem{lem}[theo]{Lemma}

\renewenvironment{proof}{\noindent{\bf Proof.}}{\qed}

\newcommand{\point}{\!\scriptscriptstyle{\bullet}}

\usepackage{perpage}
\MakePerPage{footnote}

\begin{document}

\renewcommand{\contentsname}{Contents}
\renewcommand{\refname}{\textbf{References}}
\renewcommand{\abstractname}{Abstract}

\begin{center}
\begin{Huge}
A Self-Interaction Leading\medskip

to Fluctuations of Order $n^{5/6}$
\end{Huge}\bigskip\bigskip \bigskip \bigskip

\begin{Large} Matthias Gorny \end{Large} \smallskip
 
\begin{large} {\it Universit\'e Paris Sud \emph{and} ENS Paris} \end{large} \bigskip \bigskip %\bigskip

\end{center}
\bigskip \bigskip \bigskip

\begin{abstract}
\noindent In~\cite{CerfGorny}, we built and studied a Curie-Weiss model exhibiting self-organized criticality : it is a model with a self-interaction leading to fluctuations of order $n^{3/4}$ and a limiting law proportional to $\exp(-x^4/12)$. In this paper we modify our model  in order to \og kill the term $x^4$ \fg{} and to obtain a self-interaction leading to fluctuations of order $n^{5/6}$ and a limiting law $C\,\exp(-\lambda x^6)\,dx$, for suitable positive constants $C$ and $\lambda$.
\end{abstract}
\bigskip \bigskip \bigskip \bigskip \bigskip



\newpage

\section{Introduction}
\label{Intro}

\noindent This paper is a sequel to the articles~\cite{CerfGorny} and~\cite{Gorny3}, in which we built and studied a Curie-Weiss model exhibiting self-organized criticality. It was the model given by the distribution
\[\frac{1}{Z_{n}}\exp\left(\frac{1}{2}\frac{(x_{1}+\dots+x_{n})^{2}}{x_{1}^{2}+\dots+x_{n}^{2}}\right)\ind{\{x_{1}^{2}+\dots+x_{n}^{2}>0\}}\,\prod_{i=1}^{n}d\r(x_{i})\,,\]
where $Z_n$ is a renormalisation factor. We proved rigorously that this model exhibits a simple phenomenon of self-organized criticality : if we build the model with a symmetric probability $\r$ on $\R$ satisfying some integrability conditions, then the sum $S_n$ of the random variables behaves as in the critical generalized Ising Curie-Weiss model (see~\cite{EN}). More precisely, the fluctuations of $S_n$ are of order $n^{3/4}$ and the limiting law is
\[\left(\frac{4}{3}\right)^{1/4}\G\left(\frac{1}{4}\right)^{-1} \exp\left(-\frac{s^{4}}{12}\right)\,ds\,.\]

\noindent The purpose of this article is to \og kill the term $x^4$ \fg{}. We modify the distribution we studied in~\cite{CerfGorny} and~\cite{Gorny3} in order to obtain a self-interaction leading to fluctuations of order $n^{5/6}$ and a limiting law
\[C\,\exp(-\lambda x^6)\,dx\,,\]
where $C$ and $\lambda$ are some positive constants.
\medskip

\noindent To this end, we first focus on the reasons why the fluctuations of $S_n$ in the model we studied in~\cite{CerfGorny} are of order $n^{3/4}$. The interacting term of the model is
\[F\big(x_{1}+\dots+x_{n},x_{1}^{2}+\dots+x_{n}^{2}\big)\,,\]
where $F(x,y)=x^2/(2y)$ for $(x,y) \in \R\times \,]0,+\infty[$. Let $I$ be the rate function for the large deviations of
\[\frac{1}{n}\sum_{k=1}^n (X_k,X_k^2),\quad n\geq 1\,,\]
where $(X_k)_{k\geq 1}$ is a sequence of independent random variables with common law~$\r$. By analysing the proofs in~\cite{CerfGorny}, we can see that the fluctuations of $S_n$ are of order $n^{3/4}$ because, in the expansion of the function $I-F$ around its minimum, the first non-vanishing term with the variable $x$ (corresponding to $S_n/n$) appears in the fourth order. More precisely, if $\s^2$ denotes the variance of $\r$ and  $\mu_4$ its fourth moment, this term is $\mu_4 x^4/(12\s^8)$.\medskip

\noindent As a consequence, in order to \og kill the term $x^4$ \fg{}, we are going to modify the interacting function $F$ of our model into some function $H$ so that, in the expansion of the function $I-H$ around its minimum, the first term with the variable $x$ only appears in the sixth order. We could consider
\[H(x,y)=F(x,y)+\frac{\mu_4 x^4}{12 \s^8}, \quad (x,y) \in \R\times \,]0,+\infty[\,.\]
However we want to build a self-interaction, thus we estimate $\mu_4$ by $(x_1^4+\dots+x_n^4)/n$ (as we estimated $\s^2$ by $(x_1^2+\dots+x_n^2)/n$ in order to build the model in~\cite{CerfGorny}). That is why the interacting term we want to consider is
\begin{multline*}
H\left(x_1+\cdots+x_n,\,x_1^2+\cdots+x_n^2,\,x_1^4+\cdots+x_n^4\right)\\
=\frac{1}{2}\frac{(x_{1}+\dots+x_{n})^{2}}{x_{1}^{2}+\dots+x_{n}^{2}}+
\frac{1}{12}\frac{(x^4_{1}+\dots+x^4_{n})(x_{1}+\dots+x_{n})^{4}}{(x_{1}^{2}+\dots+x_{n}^{2})^4}\,.
\end{multline*}
We observe with computer simulations that, with this interacting term, for several probability measures $\r$, the fluctuations of the sum $S_n$ are of order $n^{5/6}$ and the limiting law is proportional to $\exp(-\lambda x^6)$ for some $\lambda>0$.\medskip

\noindent In sections~\ref{Preliminaries} and~\ref{ExpansionI*} we initiate the proof of a fluctuations theorem for $S_n$ with this interacting function $H$. We use the same techniques as in~\cite{CerfGorny} : we compute the expansion of $I_{\point}-H$ where $I_{\point}$ is the rate function for the large deviations of
\[\frac{1}{n}\sum_{k=1}^n (X_k,X_k^2,X_k^4),\quad n\geq 1\,.\]

\noindent Unfortunately we encountered several problems with the rest of the proof : the techniques we used in~\cite{CerfGorny} have not been successful and we had to modify $H$. Our investigations to build an interacting function $H$ leading to fluctuations of order $n^{5/6}$ and amenable to a mathematical analysis led us to consider the following model :\medskip

\noindent{\bf The model.} \textit{Let $\r$ be a probability measure on $\R$ which is not the Dirac mass at $0$. Let $H$ be the function given by
\[\forall (x,y,z) \in \R\times \R\backslash\{0\}\times \R\qquad H(x,y,z)=\frac{x^{2}}{2y}+\frac{1}{12}\frac{zx^{4}y^{5}}{y^{9}+x^{10}+zx^4y^4}\,.\]
For any $n\geq 1$, we denote
\begin{multline*}
Z_{H,n}=\int_{\R^{n}}\exp\left(H\left(x_1+\cdots+x_n,\,x_1^2+\cdots+x_n^2,\,x_1^4+\cdots+x_n^4\right)\right)\\
\times\ind{\{x_{1}^{2}+\dots+x_{n}^{2}>0\}}\,\prod_{i=1}^{n}d\r(x_{i})\,.
\end{multline*}
We consider $(X^{n}_{k})_{1\leq k \leq n}$ an infinite triangular array of real-valued random variables such that, for all $n \geq 1$, $(X^{1}_{n},\dots,X^{n}_{n})$ has the law $\widetilde{\mu}_{H,n,\r}$, which is the distribution with density
\[\frac{1}{Z_{H,n}}\exp\left(H\left(x_1+\cdots+x_n,\,x_1^2+\cdots+x_n^2,\,x_1^4+\cdots+x_n^4\right)\right)\ind{\{x_{1}^{2}+\dots+x_{n}^{2}>0\}}\]
with respect to $\r^{\otimes n}$. We denote
\[S_{n}=X^{1}_{n}+\dots+X^{n}_{n}, \quad T_{n}=(X^{1}_{n})^{2}+\dots+(X^{n}_{n})^{2} \quad \mbox{and} \quad U_{n}=(X^{1}_{n})^{4}+\dots+(X^{n}_{n})^{4}.\]}

\noindent This model is well-defined : $Z_{H,n}$ is finite for any $n\geq 1$. Indeed
\[\forall (x,y,z) \in \R\times \R\backslash\{0\}\times \R\qquad H(x,y,z)\leq \frac{x^{2}}{2y}+\frac{zx^{4}}{12y^{4}}=\frac{x^{2}}{2y}+\left(\frac{x^{2}}{y}\right)^{2}\frac{z}{12y^2}\,.\]
We have
\[\forall (x_{1},\dots,x_{n})\in \R^{n}\qquad \left(\sum_{i=1}^{n}x_{i}^{2}\right)^{2} \geq \,\sum_{i=1}^{n}x_{i}^{4}\]
and, by convexity of the function $t\longmapsto t^2$, we get
\[\forall n \geq 1 \qquad 1\leq Z_{H,n} \leq \exp\left(\frac{n}{2}+\frac{n^{2}}{12}\right)<+\infty\,.\]

\noindent We state next our main result :

\begin{theo} Let $\r$ be a symmetric probability measure on $\R$ whose support contains at least five points and such that
\[ \exists w_0>0 \qquad  \int_{\R} e^{w_0 z^4}\,d\r(z)<+\infty\,.\]
We denote by $\s^2$ the variance of $\r$, by $\mu_4$ its fourth moment, by $\mu_6$ its sixth moment and by $\mu_8$ its eighth moment. We assume that
\[5\mu_{4}^{2}>2\s^{2}\mu_{6}\,.\]
Then, under $\widetilde{\mu}_{H,n,\r}$, $(S_{n}/n,T_{n}/n,U_n/n)$ converges in probability to $(0,\s^{2},\mu_4)$.\medskip

\noindent Moreover, if $\r$ admits a bounded density with respect to the Lebesgue measure on $\R$, then, under $\widetilde{\mu}_{H,n,\r}$,
\[\left(\frac{\mu_4^2}{\s^2}-\frac{2\mu_6}{5}\right)^{1/6}\frac{S_{n}}{\s^2n^{5/6}} \overset{\Lc}{\underset{n \to \infty}{\longrightarrow}} \left(\frac{81}{2}\right)^{1/6} \G\left(\frac{1}{6}\right)^{-1}\exp\left(-\frac{s^6}{18}\right)\,ds\,.\]
\label{MainTheo}
\end{theo}

\noindent In section~\ref{FluctuationsTheorem}.\ref{FluctuationsTheoremProof}, we will actually prove this theorem for more general interacting functions $H$ and more general probability measures $\r$.\medskip

\noindent After giving some preliminaries, simulations and notations in section~\ref{Preliminaries}, we study the smoothness of $I_{\point}$ and we compute its expansion around $(0,\s^2,\mu_4)$ in section~\ref{ExpansionI*}. Next, in section~\ref{ConstructionInteraction}, we explain the first problems we encounter and we investigate how to build an interacting term which is amenable to a mathematical analysis. Finally, in section~\ref{FluctuationsTheorem}, we give the proof of (an extended version of) theorem~\ref{MainTheo}. We end this paper by a discussion about a model with fluctuations of order $n^{1-1/2k}$ for $k\geq 4$.
\newpage

\section{Preliminaries}
\label{Preliminaries}

\noindent We denote by $F$ the function defined by  
\[\forall (x,y) \in \R\times \R\backslash\{0\} \qquad F(x,y)=\frac{x^2}{2y}\,.\]

\noindent We recall the following proposition, which is proved in section~5~of~\cite{CerfGorny} :

\begin{prop} Let $\r$ be a symmetric probability measure on $\R$ with variance $\s^2>0$ such that the function
\[\L : (u,v) \in \R^2 \longmapsto \ln \int_{\R}e^{uz+vz^2}\,d\r(z)\]
is finite in the neighbourhood of $(0,0)$. We define $I$ by
\[\forall (x,y)\in \R^2 \qquad I(x,y)=\sup_{(u,v)\in \R^2}\,\left(ux+vy-\L(u,v)\right)\,.\]
Then the function $I-F$ has a unique minimum on $\R\times \R\backslash\{0\}$ at $(0,\s^2)$, with $(I-F)(0,\s^2)=0$. Moreover, if the support of $\r$ contains at least three points and if $\mu_4$ denotes the fourth moment of $\r$, then, when $(x,y)$ goes to $(0,\s^2)$,
\[I(x,y)-F(x,y) \sim \frac{\mu_4 x^4}{12\s^8}+\frac{(y-\s^2)^4}{2(\mu_4-\s^4)}\,.\]
\label{minI-F}
\end{prop}

\noindent This is the starting point for the construction of an interaction term. Indeed, as we explained in the introduction, in order to \og kill the term $x^4$ \fg{}, it is enough to add some function $R$ to $F$ so that the term $\mu_4 x^4/(12\s^8)$ vanishes from the above expansion and so that
\[I(x,y)-(F+R)(x,y) \sim Ax^6+\frac{(y-\s^2)^4}{2(\mu_4-\s^4)}\,,\]
for some $A>0$. However we want to build a self-interaction, thus we have to estimate $\mu_4$ by $(x_1^4+\dots+x_n^4)/n$ (as we estimated $\s^2$ by $(x_1^2+\dots+x_n^2)/n$ in order to build our model in~\cite{CerfGorny}). Hence it seems natural to consider $H=F+R$, with
\[R: (x,y,z)\in \R\times \R\backslash\{0\}\times\R\longmapsto \frac{zx^{4}}{12y^{4}}\,,\]
and this leads us to study the rate function $I_{\point}$ of the large deviations for $\widetilde{\nu}_{\point n,\r}$, the law of $(S_{n}/n,T_{n}/n,U_{n}/n)$ under $\r^{\otimes n}$. 
\medskip

\noindent For $n\geq 1$ and $H=F+R$, let us consider $S_n=X_1^n+\dots+X_n^n$, where the law of $(X_1^n,\dots,X_n^n)$ has the density
\[(x_1,\dots,x_n)\longmapsto \frac{1}{Z_{H,n}}\exp\left(H\left(x_1+\cdots+x_n,\,x_1^2+\cdots+x_n^2,\,x_1^4+\cdots+x_n^4\right)\right)\]
with respect to $\r^{\otimes n}$. We made computer simulations of this model which support us in the choice of $H=F+R$. We used Metropolis-within-Gibbs algorithms (cf. section~4~of~\cite{RR}) and obtain :
\newpage

\includegraphics[height=6cm]{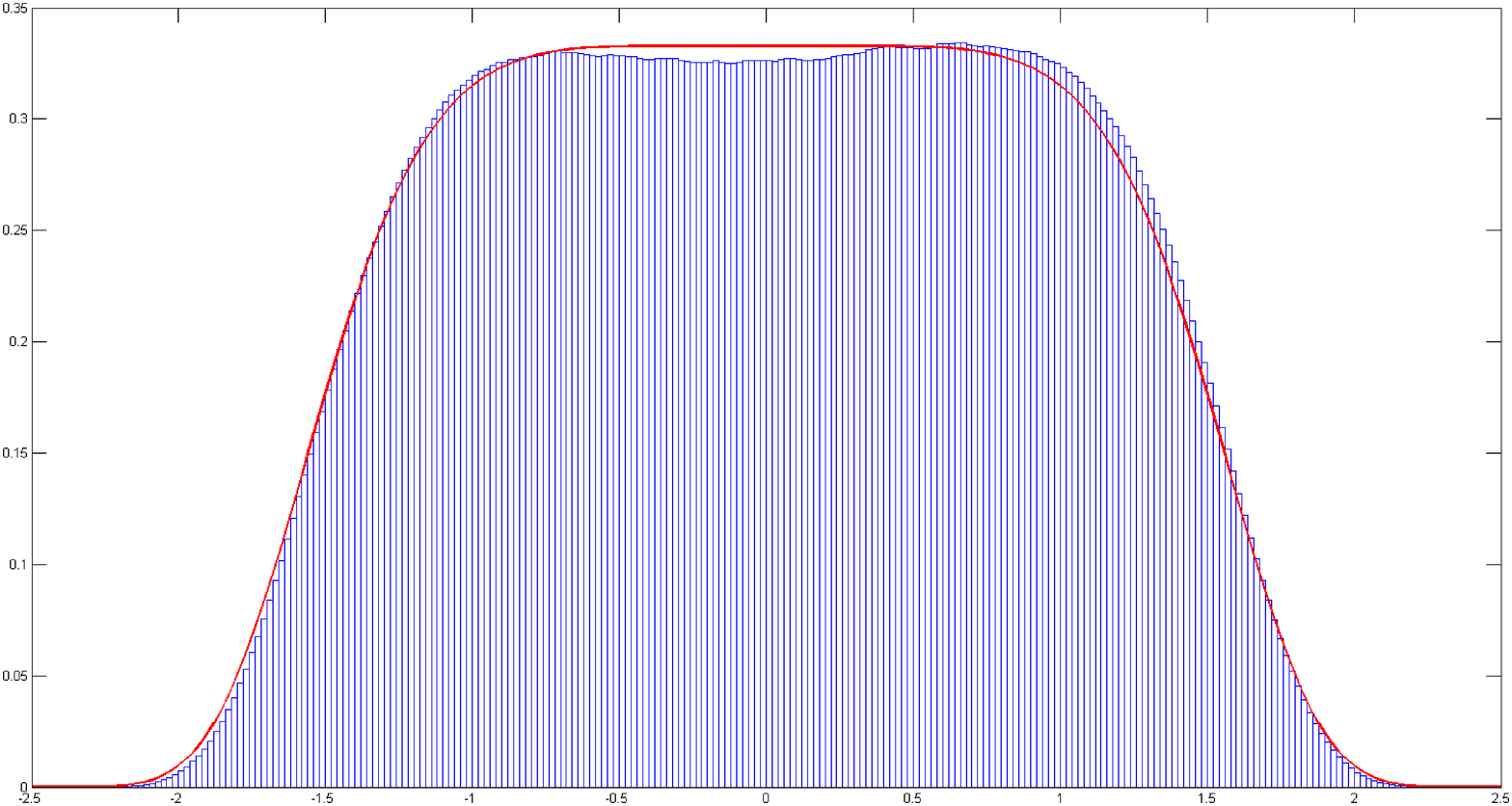}

\begin{center}
\textsc{In blue, the renormalized histogram of $6,17\times 10^{11}$ simulations of $S_n/n^{5/6}$, for $n=10000$ and $\r$ having a density proportional to $x\longmapsto\exp(-x^4)$. In red, the graph of the density function
\[x\longmapsto\left(\frac{81}{2}\right)^{1/6} \G\left(\frac{1}{6}\right)^{-1}\,\exp\left(-\frac{x^6}{18}\right)\,.\]}
\end{center}

\noindent We end this section by giving some notations. For a symmetric probability measure $\r$ on $\R$ which is not the Dirac mass at~$0$, we denote by $\nu_{\point \r}$ the law of $(Z,Z^{2},Z^{4})$ when $Z$ is a random variable with distribution $\r$. We define the Log-Laplace $\L_{\point }$ of $\nu_{\point \r}$ by
\[\forall(u,v,w) \in \R^{3} \qquad \L_{\point }(u,v,w)=\ln \int_{\R}e^{uz+vz^{2}+wz^{4}}\,d\r(z)\,.\]
If $\L_{\point }$ is finite in a neighbourhood of $(0,0,0)$ then the Cram\'er theorem (cf.~\cite{DZ}) states that $(\widetilde{\nu}_{\point n,\r})_{n \geq 1}$ satisfies the large deviations principle with speed $n$, governed by the Cram\'er transform $I$ of $\nu_{\point \r}$ defined by
\[\forall (x,y,z) \in \R^3 \qquad I_{\point }(x,y,z)=\sup_{(u,v,w)\in \R^{3}}\left(xu+yv+zw-\L_{\point }(u,v,w)\right)\,.\]
We denote by $D_{\L_{\point }}$ and $D_{I_{\point }}$ the domains of $\R^3$ where the functions $\L_{\point }$ et $I_{\point }$ are finite. We introduce next the subsets of $\R^3$
\[\Theta=\{\,(x,y,z) \in \R^3 : x^2\leq y,\, y^2\leq z\} \qquad\mbox{and} \qquad \Theta^{\!*}=\Theta \cap (\R\times \R\backslash\{0\}\times\R)\,.\]
By convexity, we have that $\widetilde{\nu}_{\point n,\r}(\Theta)=1$. We get that, under $\widetilde{\mu}_{\point n,\r}$, the distribution of $(S_{n}/n,T_{n}/n,U_{n}/n)$ is
\[\frac{1}{Z_{\point n}}\exp\left(H(nx,ny,nz)\right)\ind{\Theta^{\!*}}(x,y,z)\,d\widetilde{\nu}_{\point n,\r}(x,y,z)\,.\]

\includegraphics[height=6cm]{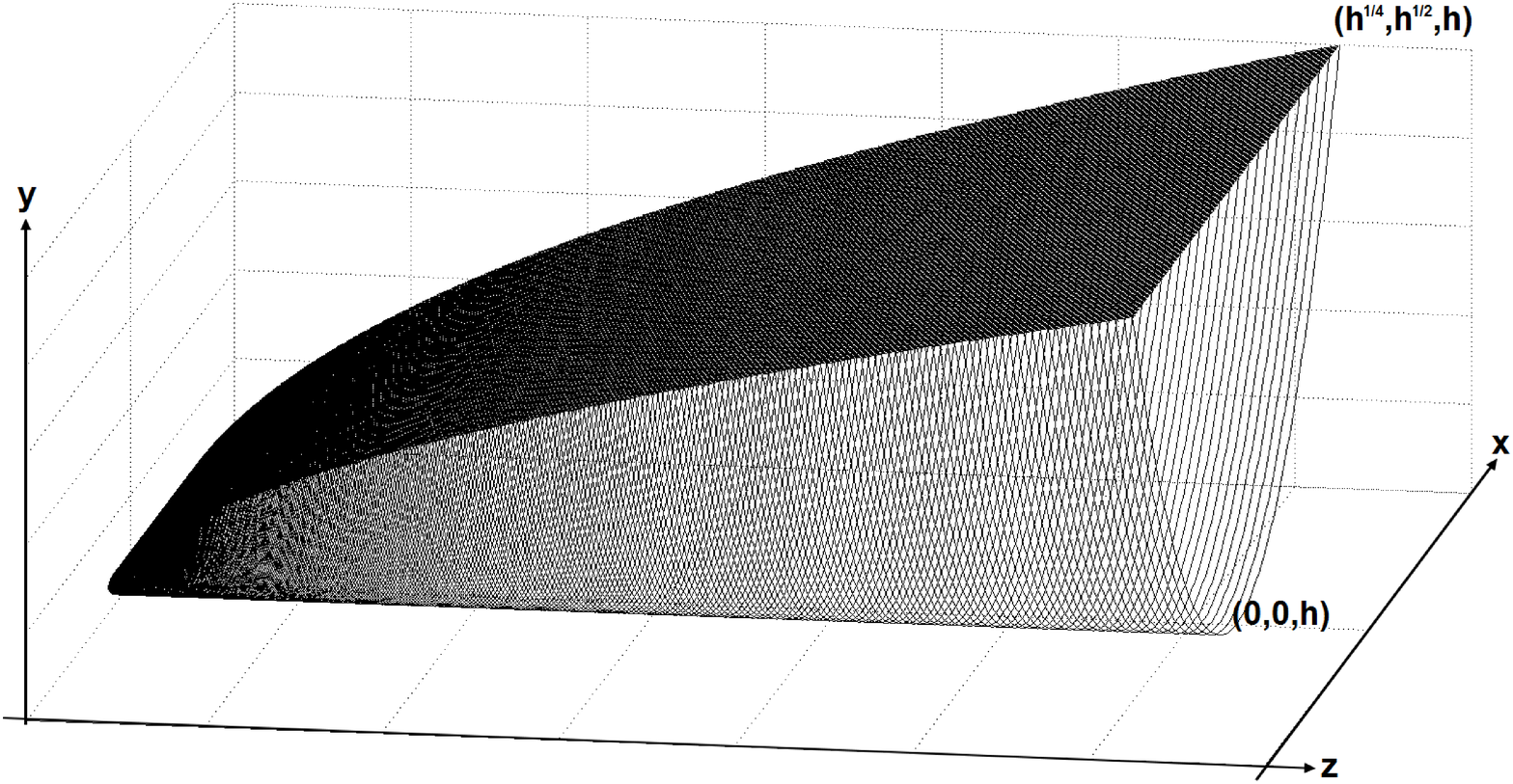}

\begin{center}
\textsc{Two views of the set of the points $(x,y,z)\in \partial \Theta$ such that $z\leq h$.}
\end{center}

\includegraphics[height=7cm]{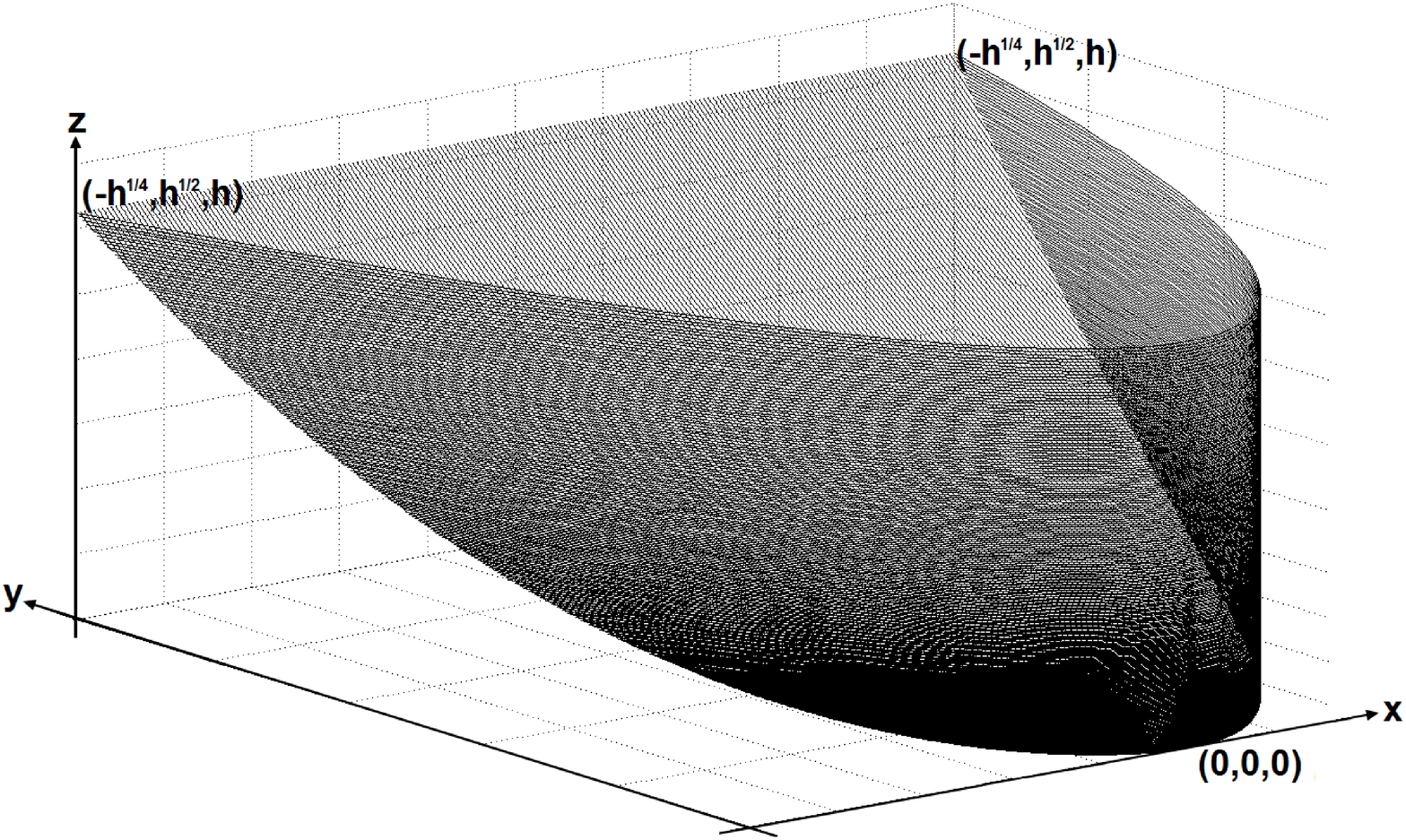}\medskip

\noindent We will proceed as we did in the article~\cite{CerfGorny}, i.e., we will study, for any $n \geq 1$, the function
\[G_n : (x,y,z) \longmapsto I_{\point }(x,y,z)-\frac{1}{n}H(nx,ny,nz)\,.\]
The Cram\'er transform $I_{\point }$ has a unique minimum at $(0,\s^2,\mu_4)$ and the method we used in the section~5.b)~of~\cite{CerfGorny} allows us to compute the expansion of $I_{\point }$ around its minimum.\medskip

\noindent In order to apply the Laplace's method, as in the section~7~of~\cite{CerfGorny}, we want to build $H$ so that $G_n$ also has a unique minimum at $(0,\s^2,\mu_4)$ for any $n\geq 1$, and so that its expansion around this minimum has the desired form :
\[Ax^6+q(y-\s^2,z-\mu_4)\,,\]
with $A>0$ and $q$ a positive definite quadratic form on $\R^2$.

\section{Expansion of $I_{\point}$ around $(0,\s^2,\mu_4)$}
\label{ExpansionI*}

\noindent Let $\r$ be a symmetric probability measure on $\R$ with variance $\s^2>0$ and such that $(0,0,0) \in \Dro_{\L_{\point }}$. In this section, we first study the smoothness of $I_{\point }$, then we compute its expansion around its minimum $(0,\s^2,\mu_4)$. In the last subsection we give the expansion of $I_{\point }-F-R$ around $(0,\s^2,\mu_4)$.

\subsection{Smoothness of $I_{\point }$}

\noindent The function $\L_{\point }$ is finite in a neighbourhood of $(0,0,0)$ thus each moment of $\r$ is finite and the covariance matrix of $\nu_{\point \r}$ is
\[\left(\begin{matrix}
\s^2 & 0 & 0  \\
0 & \mu_{4}-\s^4 & \mu_{6}-\s^2\mu_{4}\\
0 & \mu_{6}-\s^2\mu_{4} & \mu_{8}-\mu_{4}^{2}
\end{matrix}\right)\,.\]\smallskip

\begin{lem} We assume that $\r$ is a symmetric probability measure on $\R$ whose support contains at least five points. Then the support of $\nu_{\point \r}$ is not included in a hyperplane of~$\R^3$ ans thus
\[(\mu_4 - \s^4)(\mu_8-\mu_4^2)\neq(\mu_6-\s^2\mu_4)^2\,.\]
\label{Support5}
\end{lem}
\vspace*{-0.3cm}

\begin{proof} Since $\r$ is symmetric, its support contains the points $a,-a,b$ and $-b$ for some $a\neq b$. Therefore the support of $\nu_{\point \r}$ contains the points 
\[(a,a^2,a^4),\quad (-a,a^2,a^4),\quad (b,b^2,b^4)\quad\mbox{and}\quad (-b,b^2,b^4)\,.\]
We observe that these four points belong to the same plane $\Pc$ whose equation~is
\[-(a^2+b^2)y+z+a^2b^2=0\,.\]
If $c$ is a fifth point in the support of $\r$ then
\[-(a^2+b^2)c^2+c^4+a^2b^2=(c^2-a^2)(c^2-b^2)\neq 0\,.\]
Thus the point $(c,c^2,c^4)$, which is in the support of $\nu_{\point \r}$, is not included in $\Pc$. Hence the support of $\nu_{\point \r}$ is not included in a hyperplane of~$\R^3$. As a consequence the covariance matrix of $\nu_{\point \r}$ is invertible (see section III.5 of [14] for a proof), i.e., $(\mu_4 - \s^4)(\mu_8-\mu_4^2)\neq(\mu_6-\s^2\mu_4)^2$.
\end{proof}\medskip

\noindent We assume next that the support of $\r$ contains at least five points. The previous lemma and the proposition~A.4\footnote{Actually it is proposition 10 of the ARXIV version of~\cite{CerfGorny}.}~of~\cite{CerfGorny} imply that $\nabla \L_{\point }$ is a $\Ck{\infty}$-diffeomorphism from $\Dro_{\L_{\point }}$ to $A_{I_{\point }}$, the admissible domain of $I_{\point }$. Moreover $A_{I_{\point }} \subset \Theta^{\!*}$ and
\[(0,\s^{2},\mu_{4})=\nabla \L_{\point }(0,0,0) \in \nabla \L_{\point }(\Dro_{\L_{\point }})=A_{I_{\point }}\,.\]
The function $I_{\point }$ is $\Ck{\infty}$ on $A_{I_{\point }}$ and, if $(x,y,z)\longmapsto (u(x,y,z),v(x,y,z),w(x,y,z))$ denotes the inverse function of $\nabla \L_{\point }$, then, for any $(x,y,z) \in A_{I_{\point }}$,
\begin{multline*}
I_{\point }(x,y,z)=xu(x,y,z)+yv(x,y,z)+zw(x,y,z)\\-\L_{\point }(u(x,y,z),v(x,y,z),w(x,y,z))\,,
\end{multline*}
\[\nabla I_{\point }(x,y,z)=(\nabla \L_{\point })^{-1}(x,y,z) = (u(x,y,z),v(x,y,z),w(x,y,z))\,,\]
\[\mathrm{D}^{2}_{(x,y,z)}I_{\point }=\left(\mathrm{D}^{2}_{(u(x,y,z),v(x,y,z),w(x,y,z))}\L\right)^{-1}\,.\]

\noindent In order to compute the derivatives of the previous terms, as in section~5.b)~of~\cite{CerfGorny}, we introduce the functions $f_j$ defined by
\[\forall j\in \N \quad \forall (u,v,w) \in \Dro_{\L_{\point }} \qquad f_{j}(u,v,w)=\frac{\displaystyle{\int_{\R}x^{j}e^{ux+vx^{2}+wx^{4}}\,d\r(x)}}{\displaystyle{\int_{\R}e^{ux+vx^{2}+wx^{4}}\,d\r(x)}}\,.\]
The functions $f_{j}$, $j \in \N$, are $\Ck{\infty}$ on $\Dro_{\L_{\point }}$ and they verify the following properties:\\
$\star$ $f_{0}$ is the identity function on $\R^{3}$ and
\[f_{1}=\frac{\partial \L}{\partial u}\,, \qquad f_{2}=\frac{\partial \L}{\partial v}  \qquad \mbox{and} \qquad f_{4}=\frac{\partial \L}{\partial w}\,.\]
$\star$ For all $j \in \N$, $f_{j}(0,0,0)=\mu_{j}$ is the $j$-th moment of $\r$. It is null if $j$ is odd, since $\r$ is symmetric. Moreover, for any $j\in \N$,
\[\frac{\partial f_{j}}{\partial u}=f_{j+1}-f_{j}f_{1}\,,\qquad \frac{\partial f_{j}}{\partial v}=f_{j+2}-f_{j}f_{2} \qquad \mbox{and} \qquad \frac{\partial f_{j}}{\partial w}=f_{j+4}-f_{j}f_{4}\,.\]

\noindent For any $(u,v,w) \in \Dro_{\L_{\point }}$, we have
\[\mathrm{D}^{2}_{(u,v,w)}\L_{\point }=\left(\begin{matrix}
f_{2}-f_{1}^{2} & f_{3}-f_{1}f_{2} & f_{5}-f_{4}f_{1}  \\
f_{3}-f_{1}f_{2} & f_{4}-f_{2}^{2} & f_{6}-f_{4}f_{2}\\
f_{5}-f_{4}f_{1} & f_{6}-f_{4}f_{2} & f_{8}-f_{4}^{2}
\end{matrix}\right)(u,v,w)\,.\]
We define
\begin{multline*}
g=(f_{2}-f_{1}^{2})(f_{4}-f_{2}^{2})(f_{8}-f_{4}^{2})+2(f_{3}-f_{1}f_{2})(f_{6}-f_{4}f_{2})(f_{5}-f_{4}f_{1})\\
-(f_{4}-f_{2}^{2})(f_{5}-f_{4}f_{1})^{2}-(f_{2}-f_{1}^{2})(f_{6}-f_{4}f_{2})^{2}-(f_{8}-f_{4}^{2})(f_{3}-f_{1}f_{2})^{2}\,.
\end{multline*}
This is a function which is positive on $\Dro_{\L_{\point }}$. Therefore
\[\forall (x,y,z) \in A_{I_{\point }} \qquad \mathrm{D}^{2}_{(x,y,z)}I_{\point }=K(u(x,y,z),v(x,y,z),w(x,y,z))\,,\]
where $K$ is a function from $\R^{3}$ to $\Sc_{3}(\R)$, the set of the symmetric matrices of size $3$, such that 
\[K_{1,1}=\frac{(f_{4}-f_{2}^{2})(f_{8}-f_{4}^{2})-(f_{6}-f_{4}f_{2})^{2}}{g}\,,\]
\[K_{2,2}=\frac{(f_{2}-f_{1}^{2})(f_{8}-f_{4}^{2})-(f_{5}-f_{4}f_{1})^{2}}{g}\,,\]
\[K_{3,3}=\frac{(f_{2}-f_{1}^{2})(f_{4}-f_{2}^{2})-(f_{3}-f_{1}f_{2})^{2}}{g}\,,\]
\[K_{1,2}=K_{2,1}=\frac{(f_{5}-f_{4}f_{1})(f_{6}-f_{4}f_{2})-(f_{3}-f_{1}f_{2})(f_{8}-f_{4}^{2})}{g}\,,\]
\[K_{1,3}=K_{3,1}=\frac{(f_{3}-f_{1}f_{2})(f_{6}-f_{4}f_{2})-(f_{5}-f_{4}f_{1})(f_{4}-f_{2}^{2})}{g}\,,\]
\[K_{2,3}=K_{3,2}=\frac{(f_{3}-f_{1}f_{2})(f_{5}-f_{4}f_{1})-(f_{2}-f_{1}^{2})(f_{6}-f_{4}f_{2})}{g}\,.\]

\subsection{Computation of the terms of the expansion of $I_{\point }$}
\label{ExpansionI*sub}

\noindent Notice that $g(0,0,0)=a\s^{2}$ with
\[a=(\mu_4 - \s^4)(\mu_8-\mu_4^2)-(\mu_6-\s^2\mu_4)^2>0\,.\]
We have
\[\mathrm{D}^{2}_{(0,\s^{2},\mu_{4})}I_{\point }=\left(\begin{matrix}
1/\s^{2} & 0 & 0  \\
0 & (\mu_{8}-\mu_{4}^{2})/a & (\mu_{4}\s^{2}-\mu_{6})/a\\
0 & (\mu_{4}\s^{2}-\mu_{6})/a & (\mu_{4}-\s^{4})/a
\end{matrix}\right)\,.\]
Let $q$ be the positive definite quadratic form on $\R^2$ given by
\[\forall (y,z) \in \R^2 \qquad q(y,z)=\frac{\mu_{8}-\mu_{4}^{2}}{2a}y^{2}+\frac{\mu_{4}\s^{2}-\mu_{6}}{a}yz+\frac{\mu_{4}-\s^{4}}{2a}z^{2}\,.\]

\noindent Taylor formula implies that, at the order $6$, the expansion of $I_{\point }$ in the neighbourhood of $(0,\s^{2},\mu_{4})$ is \begin{multline*}
I_{\point }(x,y,z)=\frac{x^{2}}{2\s^{2}}+q(y-\s^{2},z-\mu_{4})\\
+\sum_{(\a,\b,\g) \in \Tc} \frac{1}{\a!\b!\g!}\,\frac{\partial^{\a+\b+\g} I_{\point }}{\partial x^{\a}\partial y^{\b}\partial z^{\g}}(0,\s^{2},\mu_{4})\,x^{\a}(y-\s^2)^{\b}(z-\mu_4)^{\g}\\
+o(\|x,y-\s^{2},z-\mu_{4}\|^{6})\,,
\end{multline*}
with
\[\Tc=\{\,(\a,\b,\g)\in \N^3 : \a+\b+\g \in \{3,4,5,6\} \}\,.\]

\noindent Thus we have to compute the terms
\[\frac{\partial^{\a+\b+\g} I_{\point }}{\partial x^{\a}\partial y^{\b}\partial z^{\g}}(0,\s^{2},\mu_{4})\]
for $(\a,\b,\g) \in \Tc$. In order to optimize the computations, we will first determine the terms of the expansion of $I_{\point }$ which are negligible compared to the term $Ax^{6}+q(y-\s^{2},z-\mu_{4})$ with $A>0$.

\subsubsection{The non-negligible terms}

\begin{lem} Let $A>0$ and $q$ be a positive definite quadratic form on $\R^2$. Then, in a neighboorhood of $(0,0,0)$,
\[\|x,y,z\|^{6}=O(Ax^{6}+q(y,z))\,.\]
Moreover, for any $(\a, \b, \g) \in \N^3$, we have
\[\frac{\a}{3}+\b+\g>2 \Longrightarrow \lim_{(x,y,z)\to(0,0,0)}\frac{x^{\a}y^{\b}z^{\g}}{Ax^{6}+q(y,z)}=0\,.\]
\label{negligeable}
\end{lem}

\begin{proof} For any $(x,y,z) \in \R^{3}\backslash \{(0,0,0)\}$, there exists a unique $(r, \theta, \phi)$ which belongs to $]0,+\infty[\times [0,2\pi[\times [0,\pi]$ and satisfies
\[\left\{\begin{array}{l}
x^{3}=r\,\mathrm{sin}\phi\,, \\
y=r\,\mathrm{cos}\theta\,\mathrm{cos}\phi\,,\\
z=r\,\mathrm{sin}\theta\,\mathrm{cos}\phi\,.
\end{array}\right.\]
Thus
\[Ax^{6}+q(y,z)=Ar^{2}\,\mathrm{sin}^{2}\phi+r^{2}\,\mathrm{cos}^{2}\phi\,q(\mathrm{cos}\theta,\mathrm{sin}\theta)\,.\]
However the set $\{\,(\mathrm{cos} \theta,\mathrm{sin} \theta) : \theta \in [0,2\pi[\,\}$ is compact in $\R^{2}$ and the continuous function $q$ is positive on this set. As a consequence $q$ has a minimum $m>0$ and a maximum $M>m$. Hence
\[\min(A,m)\,r^{2} \leq Ax^{6}+q(y,z) \leq \max(A,m)\,r^{2}\,.\]
We get that
\[\frac{\|x,y,z\|^{6}}{Ax^{6}+q(y,z)}=\frac{\left(x^2+y^2+z^2\right)^{3}}{Ax^{6}+q(y,z)}\leq\frac{\left(r^{2/3}\mathrm{sin}^{2/3}\phi  +r^2\mathrm{cos}^2\phi\right)^{3}}{\min(A,m)\,r^{2}}=\frac{\left(1+r^{4/3}\right)^{3}}{\min(A,m)}\,.\]
This is a bounded quantity when $r$ tends to $0$. Next
\[\left|\frac{x^{\a}y^{\b}z^{\g}}{Ax^{6}+q(y,z)}\right|\leq \frac{|r\,\mathrm{sin}\phi|^{\a/3} |r\,\mathrm{cos}\theta\,\mathrm{cos}\phi|^{\b} |r\,\mathrm{sin}\theta\,\mathrm{cos}\phi|^{\g}}{\min(A,m)\,r^{2}}=O(r^{\a/3+\b+\g-2})\,.\]
Since the convergence of $(x,y,z)$ to $(0,0,0)$ is equivalent to the convergence of $r$ to $0$, the lemma is proved.
\end{proof} \medskip

\noindent This lemma states that the terms $x^{\a}y^{\b}z^{\g}$, $(\a, \b, \g) \in \Tc$, which are not negligible at $(0,\s^{2},\mu_{4})$ compared to $Ax^{6}+q(y-\s^{2},z-\mu_{4})$, are such that
\[\frac{\a}{3}+\b+\g\leq~2\,.\]
Thus, these terms are those for which $(\a, \b, \g)$ is $(2,1,0)$, $(2,0,1)$, $(3,0,0)$, $(3,1,0)$, $(3,0,1)$, $(4,0,0)$, $(5,0,0)$ or $(6,0,0)$. Let us compute the coefficients of these terms in the expansion of $I_{\point }$. We denote
\[\begin{array}{c}
k_{1}=(f_{4}-f_{2}^{2})(f_{8}-f_{4}^{2})-(f_{6}-f_{4}f_{2})^{2}\,,\\[0.2cm]
k_{2}=(f_{5}-f_{4}f_{1})(f_{6}-f_{4}f_{2})-(f_{3}-f_{1}f_{2})(f_{8}-f_{4}^{2})\,,\\[0.2cm]
k_{3}=(f_{3}-f_{1}f_{2})(f_{6}-f_{4}f_{2})-(f_{5}-f_{4}f_{1})(f_{4}-f_{2}^{2})\,.
\end{array}\]

\subsubsection{The terms at the third order}

\noindent Let us start with the terms at third order which might be non-negligible compared to $Ax^{6}+q(y-\s^{2},z-\mu_{4})$ :
\begin{align*}
\frac{\partial^{3}I_{\point }}{\partial x^{2} \partial y}&=\frac{\partial}{\partial x}\left(\frac{\partial^{2}I_{\point }}{\partial x \partial y}\right)=\frac{\partial K_{1,2}(u,v,w)}{\partial x}\\
&=\frac{\partial u}{\partial x}\times \frac{\partial K_{1,2}}{\partial u}(u,v,w)+\frac{\partial v}{\partial x}\times \frac{\partial K_{1,2}}{\partial v}(u,v,w)+\frac{\partial w}{\partial x}\times \frac{\partial K_{1,2}}{\partial w}(u,v,w)\,.
\end{align*}
We have
\[\frac{\partial v}{\partial x}(0,\s^{2},\mu_{4})=\frac{\partial^{2} I_{\point }}{\partial x\partial y}(0,\s^{2},\mu_{4})=0=\frac{\partial^{2} I_{\point }}{\partial x\partial z}(0,\s^{2},\mu_{4})=\frac{\partial w}{\partial x}(0,\s^{2},\mu_{4})\,,\]
\[\frac{\partial u}{\partial x}(0,\s^{2},\mu_{4})=\frac{\partial^{2} I_{\point }}{\partial x^{2}}(0,\s^{2},\mu_{4})=\frac{1}{\s^{2}}\,,\]
thus
\[\frac{\partial^{3}I_{\point }}{\partial x^{2} \partial y}(0,\s^{2},\mu_{4})=\frac{1}{\s^{2}}\left( \frac{1}{g(0,0,0)}\frac{\partial k_{2}}{\partial u}(0,0,0)-\frac{k_{2}(0,0,0)}{g^{2}(0,0,0)}\frac{\partial g}{\partial u}(0,0,0)  \right)\,.\]
We have $k_{2}(0,0,0)=k_{3}(0,0,0)=0$ and $g(0,0,0)=\s^{2}k_{1}(0,0,0)$ with
\[k_{1}(0,0,0)=\mu_{8}\mu_{4}-\mu_{4}^{3}-\mu_{8}\s^{4}-\mu_{6}^{2}+2\mu_{6}\mu_{4}\s^{2}\,.\]
Using the properties of the functions $f_{i}$, $i \in \N$, for computing their partial derivatives, we get
\[\frac{\partial k_{2}}{\partial u}(0,0,0)=-k_{1}(0,0,0)\,.\]
Hence
\[\frac{\partial^{3}I_{\point }}{\partial x^{2} \partial y}(0,\s^{2},\mu_{4})=\frac{-k_{1}(0,0,0)}{\s^{4}k_{1}(0,0,0)}=-\frac{1}{\s^{4}}\,.\]
We compute next that
\[\frac{\partial k_{1}}{\partial u}(0,0,0)=\frac{\partial k_{3}}{\partial u}(0,0,0)=0\,.\]
This implies that
\[\frac{\partial^{3}I_{\point }}{\partial x^{3}}(0,\s^{2},\mu_{4})=\frac{\partial^{3}I_{\point }}{\partial x^{2} \partial z}(0,\s^{2},\mu_{4})=0\,.\]
But we already knew that the third partial derivative of $I_{\point }$ with respect to $x$ is null at $(0,\s^2,\mu_4)$ since $I_{\point }$ is even in its first variable. \medskip

\noindent We have shown that
\[\frac{1}{\a!\b!\g!}\,\frac{\partial^{\a+\b+\g} I_{\point }}{\partial x^{\a}\partial y^{\b}\partial z^{\g}}(0,\s^{2},\mu_{4})=\left\{\begin{array}{cl}
\displaystyle{-\frac{1}{2\s^{4}}} & \quad \mbox{if} \quad (\a,\b,\g)=(2,1,0)\,,\\[0.2cm]
\displaystyle{0}  & \quad \mbox{if} \quad (\a,\b,\g)=(2,0,1)\,,\\[0.2cm]
\displaystyle{0}  & \quad \mbox{if} \quad (\a,\b,\g)=(3,0,0)\,.
\end{array}\right.\]

\subsubsection{The terms at the fourth order}

\noindent Let us focus now on the non-negligible terms at the fourth order :
\begin{multline*}
\frac{\partial^{4}I_{\point }}{\partial x^{4}}=\frac{\partial^{2}K_{1,1}(u,v,w)}{\partial x^{2}}=\frac{\partial^{2} u}{\partial x^{2}} \frac{\partial K_{1,1}}{\partial u}(u,v,w)+\frac{\partial^{2} v}{\partial x^{2}} \frac{\partial K_{1,1}}{\partial v}(u,v,w)\\
+\frac{\partial^{2} w}{\partial x^{2}} \frac{\partial K_{1,1}}{\partial w}(u,v,w)+\left(\frac{\partial u}{\partial x}\right)^{2} \frac{\partial^{2} K_{1,1}}{\partial u^{2}}(u,v,w)+\left(\frac{\partial v}{\partial x}\right)^{2} \frac{\partial^{2} K_{1,1}}{\partial v^{2}}(u,v,w)\\
+\left(\frac{\partial w}{\partial x}\right)^{2} \frac{\partial^{2} K_{1,1}}{\partial w^{2}}(u,v,w)+2\frac{\partial u}{\partial x}\frac{\partial v}{\partial x}\frac{\partial^{2} K_{1,1}}{\partial u \partial v}(u,v,w)\\
+2\frac{\partial u}{\partial x}\frac{\partial w}{\partial x}\frac{\partial^{2} K_{1,1}}{\partial u \partial w}(u,v,w)+2\frac{\partial v}{\partial x}\frac{\partial w}{\partial x}\frac{\partial^{2} K_{1,1}}{\partial v \partial w}(u,v,w)\,.
\end{multline*}
We have
\[\frac{\partial^{2} u}{\partial x^{2}}(0,\s^{2},\mu_{4})=\frac{\partial^{3} I_{\point }}{\partial x^{3}}(0,\s^{2},\mu_{4})=0=\frac{\partial^{3} I_{\point }}{\partial x^{2}\partial z}(0,\s^{2},\mu_{4})=\frac{\partial^{2} w}{\partial x^{2}}(0,\s^{2},\mu_{4})\,,\]
\[\frac{\partial^{2} v}{\partial x^{2}}(0,\s^{2},\mu_{4})=\frac{\partial^{3} I_{\point }}{\partial x^{2}\partial y}(0,\s^{2},\mu_{4})=-\frac{1}{\s^{4}}\,.\]
As a consequence
\[\frac{\partial^{4}I_{\point }}{\partial x^{4}}(0,\s^{2},\mu_{4})=\frac{1}{\s^{4}}\frac{\partial^{2} K_{1,1}}{\partial u^{2}}(0,0,0)-\frac{1}{\s^{4}}\frac{\partial K_{1,1}}{\partial v}(0,0,0)\,.\]
We have
\[\frac{\partial^{2} K_{1,1}}{\partial u^{2}}=\frac{1}{g} \frac{\partial^{2}k_{1}}{\partial u^{2}}-\frac{2}{g^{2}}\frac{\partial g }{\partial u}\,\frac{\partial k_{1}}{\partial u}-\frac{k_{1}}{g^{2}}\frac{\partial^{2} g }{\partial u^{2}}+\frac{2k_{1}}{g^{3}}\left(\frac{\partial g }{\partial u}\right)^{2}\,,\]
and the properties of the functions $f_{i}$, $i \in \N$, for computing their partial derivatives give us
\[\frac{\partial g}{\partial u}(0,0,0)=0\,,\]
so that
\[\frac{\partial^{2} K_{1,1}}{\partial u^{2}}(0,0,0)=\frac{1}{\s^{4}k_{1}(0,0,0)}\left(\s^{2} \frac{\partial^{2}k_{1}}{\partial u^{2}}(0,0,0)-\frac{\partial^{2} g }{\partial u^{2}}(0,0,0)\right)\,.\]
Moreover
\begin{align*}
\frac{\partial K_{1,1}}{\partial v}(0,0,0)&=\frac{1}{g(0,0,0)}\frac{\partial k_{1}}{\partial v}(0,0,0)-\frac{k_{1}(0,0,0)}{g^{2}(0,0,0)}\frac{\partial g}{\partial v}(0,0,0)\\
&=\frac{1}{\s^{4}k_{1}(0,0,0)}\left(\s^{2}\frac{\partial k_{1}}{\partial v}(0,0,0)-\frac{\partial g}{\partial v}(0,0,0)\right)\,.
\end{align*}
We compute that
\begin{multline*}
\frac{\partial k_{1}}{\partial v}(0,0,0)=\frac{\partial^{2} k_{1}}{\partial u^{2}}(0,0,0)=-\mu_{8}\mu_{6}+\mu_{10}\mu_{4}-2\mu_{8}\mu_{4}\s^{2}-\mu_{10}\s^{4}+3\mu_{8}\s^{6}\\-\mu_{6}\mu_{4}^{2} +3\mu_{4}^{3}\s^{2}+4\mu_{6}^{2}\s^{2}-6\mu_{6}\mu_{4}\s^{4}\,.
\end{multline*}
After factorising by $k_1(0,0,0)$, we get that this quantity is equal to 
\[-3\s^{2}k_{1}(0,0,0)+\mu_{10}(\mu_{4}-\s^{4})+\mu_{8}(\mu_{4}\s^{2}-\mu_{6})+\mu_{6}(\mu_{6}\s^{2}-\mu_{4}^{2})\,.\]
We compute similarly that
\begin{equation}
\frac{\partial^{2} g }{\partial u^{2}}(0,0,0)=-(\mu_{4}+4\s^{4})k_{1}(0,0,0)+\s^{2}\eta\label{Dguu},
\end{equation}
\begin{equation}
\frac{\partial g }{\partial v}(0,0,0)=(\mu_{4}-4\s^{4})k_{1}(0,0,0)+\s^{2}\eta\label{Dgv}.
\end{equation}
where $\eta=\mu_{10}(\mu_{4}-\s^{4})+\mu_{8}(\mu_{4}\s^{2}-\mu_{6})+\mu_{6}(\mu_{6}\s^{2}-\mu_{4}^{2})$.
Finally
\[\frac{\partial^{4}I_{\point }}{\partial x^{4}}(0,\s^{2},\mu_{4})=\frac{(-3\s^{4}+\mu_{4}+4\s^{4}+3\s^{4}+\mu_{4}-4\s^{4})k_{1}(0,0,0)}{\s^{8}k_{1}(0,0,0)}=\frac{2\mu_{4}}{\s^{8}}\,.\]
We have likewise
\begin{align*}
\frac{\partial^{4}I_{\point }}{\partial x^{3}\partial y}(0,\s^{2},\mu_{4})&=\frac{\partial^{2}K_{1,2}(u,v,w)}{\partial x^{2}}(0,\s^{2},\mu_{4})\\
&=\frac{1}{\s^{4}}\frac{\partial^{2} K_{1,2}}{\partial u^{2}}(0,0,0)-\frac{1}{\s^{4}}\frac{\partial K_{1,2}}{\partial v}(0,0,0)\\
&=\frac{1}{\s^{8}k_{1}(0,0,0)}\left(\s^{2} \frac{\partial^{2}k_{2}}{\partial u^{2}}-\frac{k_{2}}{k_{1}}\frac{\partial^{2} g }{\partial u^{2}}-\s^{2}\frac{\partial k_{2}}{\partial v}+\frac{k_{2}}{k_{1}}\frac{\partial g}{\partial v}\right)(0,0,0)
\end{align*}
and
\[\frac{\partial^{4}I_{\point }}{\partial x^{3}\partial z}(0,\s^{2},\mu_{4})\!=\!\frac{1}{\s^{8}k_{1}(0,0,0)}\left(\s^{2} \frac{\partial^{2}k_{3}}{\partial u^{2}}-\frac{k_{3}}{k_{1}}\frac{\partial^{2} g }{\partial u^{2}}-\s^{2}\frac{\partial k_{3}}{\partial v}+\frac{k_{3}}{k_{1}}\frac{\partial g}{\partial v}\right)(0,0,0)\,.\]
But $k_{2}(0,0,0)=k_{3}(0,0,0)=0$ and we compute that
\[\frac{\partial^{2}k_{2}}{\partial u^{2}}(0,0,0)=\frac{\partial k_{2}}{\partial v}(0,0,0)=\frac{\partial^{2}k_{3}}{\partial u^{2}}(0,0,0)=\frac{\partial k_{3}}{\partial v}(0,0,0)=0\,.\]
Hence we have shown that
\[\frac{1}{\a!\b!\g!}\,\frac{\partial^{\a+\b+\g} I_{\point }}{\partial x^{\a}\partial y^{\b}\partial z^{\g}}(0,\s^{2},\mu_{4})=\left\{\begin{array}{cl}
\displaystyle{\frac{\mu_4 x^4}{12\s^{4}}} & \quad \mbox{if} \quad (\a,\b,\g)=(4,0,0)\,,\\[0.2cm]
\displaystyle{0}  & \quad \mbox{if} \quad (\a,\b,\g)=(3,1,0)\,,\\[0.2cm]
\displaystyle{0}  & \quad \mbox{if} \quad (\a,\b,\g)=(3,0,1)\,.
\end{array}\right.\]

\subsubsection{The terms at the fifth and sixth orders}

\noindent We still have to prove that
\[\frac{1}{120}\frac{\partial^{5}I_{\point }}{\partial x^{5}}(0,\s^{2},\mu_{4})=0 \qquad \mbox{and} \qquad \frac{1}{720}\frac{\partial^{6}I_{\point }}{\partial x^{4}}(0,\s^{2},\mu_{4})=A>0\,.\]
By symmetry of $I_{\point }$ at its first variable, we obtain immediately that its fifth partial derivative with respect to $x$ is null. Let us determine its sixth partial derivative with respect to $x$. We notice first that
\[\frac{\partial^{3} u}{\partial x^{3}}(0,\s^{2},\mu_{4})=\frac{\partial^{4} I_{\point }}{\partial x^{4}}(0,\s^{2},\mu_{4})=\frac{2\mu_{4}}{\s^{8}}\,,\]
\[\frac{\partial^{3} v}{\partial x^{3}}(0,\s^{2},\mu_{4})=\frac{\partial^{4} I_{\point }}{\partial x^{3}\partial y}(0,\s^{2},\mu_{4})=0=\frac{\partial^{4} I_{\point }}{\partial x^{3}\partial z}(0,\s^{2},\mu_{4})=\frac{\partial^{3} w}{\partial x^{3}}(0,\s^{2},\mu_{4})\,,\]
\[\frac{\partial^{4} u}{\partial x^{4}}(0,\s^{2},\mu_{4})=\frac{\partial^{5} I_{\point }}{\partial x^{5}}(0,\s^{2},\mu_{4})=0\,.\]
Thus we know the partial derivatives with respect to $x$ at $(0,\s^{2},\mu_{4})$ of the functions $u$, $v$, and $w$ until the third order. We write then the sixth partial derivative $I_{\point }$ with respect to $x$, taken at $(0,\s^{2},\mu_{4})$ and we only keep the terms which do not vanish because of the symmetries:
\begin{multline*}
\frac{\partial^{6}I_{\point }}{\partial x^{6}}(0,\s^{2},\mu_{4})=\frac{\partial^{4}K_{1,1}(u,v,w)}{\partial x^{4}}(0,\s^{2},\mu_{4})=\frac{\partial^{4} v}{\partial x^{4}}(0,\s^{2},\mu_{4})\frac{\partial K_{1,1}}{\partial v}(0,0,0)\\
+\frac{\partial^{4} w}{\partial x^{4}}(0,\s^{2},\mu_{4})\frac{\partial K_{1,1}}{\partial w}(0,0,0)+3\left(\frac{\partial^{2} v}{\partial x^{2}}(0,\s^{2},\mu_{4})\right)^{2}\frac{\partial^{2} K_{1,1}}{\partial v^{2}}(0,0,0)\\
+\left(\frac{\partial u}{\partial x}(0,\s^{2},\mu_{4})\right)^{4}\frac{\partial^{4} K_{1,1}}{\partial u^{4}}(0,0,0)+4\frac{\partial^{3} u}{\partial x^{3}}(0,\s^{2},\mu_{4})\frac{\partial u}{\partial x}(0,\s^{2},\mu_{4})\frac{\partial^{2} K_{1,1}}{\partial u^{2}}(0,0,0)\\
+6\left(\frac{\partial u}{\partial x}(0,\s^{2},\mu_{4})\right)^{2}\frac{\partial^{2} v}{\partial x^{2}}(0,\s^{2},\mu_{4})\frac{\partial^{3} K_{1,1}}{\partial u^{2}\partial v}(0,0,0)\,.
\end{multline*}
We computed above that 
\[\frac{\partial K_{1,1}}{\partial v}(0,0,0)=\frac{\s^{4}-\mu_{4}}{\s^{4}}\qquad \mbox{and} \qquad \frac{\partial^{2} K_{1,1}}{\partial u^{2}}(0,0,0)=\frac{\mu_{4}+\s^{4}}{\s^{4}}\,.\]
As a consequence
\[\frac{\partial^{6}I_{\point }}{\partial x^{6}}(0,\s^{2},\mu_{4})=\frac{3}{\s^{8}}\frac{\partial^{2} K_{1,1}}{\partial v^{2}}(0,0,0)+\frac{8\mu_{4}(\mu_{4}+\s^{4})}{\s^{14}}-\frac{6}{\s^{8}}\frac{\partial^{3} K_{1,1}}{\partial u^{2}\partial v}(0,0,0)\]
\begin{equation}
+\frac{1}{\s^{8}}\frac{\partial^{4} K_{1,1}}{\partial u^{4}}(0,0,0)+\frac{\partial^{4} v}{\partial x^{4}}(0,\s^{2},\mu_{4})\frac{\s^{4}-\mu_{4}}{\s^{4}}+\frac{\partial^{4} w}{\partial x^{4}}(0,\s^{2},\mu_{4})\frac{\partial K_{1,1}}{\partial w}(0,0,0).
\label{DIxxxxxx}
\end{equation}
We have
\begin{multline*}
\frac{\partial^{4}v}{\partial x^{4}}(0,\s^{2},\mu_{4})=\frac{\partial^{3}K_{1,2}(u,v,w)}{\partial x^{3}}(0,\s^{2},\mu_{4})=\frac{\partial^{3} u}{\partial x^{3}}(0,\s^{2},\mu_{4})\frac{\partial K_{1,2}}{\partial u}(0,0,0)\\
+\left(\frac{\partial u}{\partial x}(0,\s^{2},\mu_{4})\right)^{3}\frac{\partial^{3} K_{1,2}}{\partial u^{3}}(0,0,0)+3\frac{\partial u}{\partial x}(0,\s^{2},\mu_{4})\frac{\partial^{2} v}{\partial x^{2}}(0,\s^{2},\mu_{4})\frac{\partial^{2} K_{1,2}}{\partial u\partial v}(0,0,0)\\
=\frac{1}{\s^{8}}\left(2\mu_{4}\frac{\partial K_{1,2}}{\partial u}(0,0,0)+\s^{2}\frac{\partial^{3} K_{1,2}}{\partial u^{3}}(0,0,0)-3\s^{2}\frac{\partial^{2} K_{1,2}}{\partial u\partial v}(0,0,0)\right)
\end{multline*}
and we have already computed that
\[\frac{\partial K_{1,2}}{\partial u}(0,0,0)=\frac{1}{g(0,0,0)}\frac{\partial k_{2}}{\partial u}(0,0,0)=-\frac{1}{\s^{2}}\,.\]
By differentiating and evaluating at $(0,0,0)$, we get
\begin{align*}
\frac{\partial^{2} K_{1,2}}{\partial u\partial v}(0,0,0)&=\frac{1}{g(0,0,0)}\frac{\partial^{2} k_{2}}{\partial u\partial v}(0,0,0)-\frac{1}{g^{2}(0,0,0)}\frac{\partial k_{2}}{\partial u}(0,0,0)\frac{\partial g}{\partial v}(0,0,0)\\
&=\frac{1}{\s^{4}k_{1}(0,0,0)}\left(\s^{2}\frac{\partial^{2} k_{2}}{\partial u\partial v}(0,0,0)+\frac{\partial g}{\partial v}(0,0,0)\right)\,.
\end{align*}
The properties of the functions $f_{i}$, $i \in \N$, and their partial derivatives give us
\[\frac{\partial^{2} k_{2}}{\partial u\partial v}(0,0,0)=3\s^{2}k_{1}(0,0,0)-\mu_{10}(\mu_{4}-\s^{4})-\mu_{8}(\mu_{4}\s^{2}-\mu_{6})-\mu_{6}(\mu_{6}\s^{2}-\mu_{4}^{2})\]
and, by formula~(\ref{Dgv}), we get
\[\frac{\partial^{2} K_{1,2}}{\partial u\partial v}(0,0,0)=\frac{3\s^{4}k_{1}(0,0,0)+(\mu_{4}-4\s^{4})k_{1}(0,0,0)}{\s^{4}k_{1}(0,0,0)}=\frac{\mu_{4}-\s^{4}}{\s^{4}}\,.\]
Finally, by using the fact that the partial derivative of $g$ with respect to $u$ at $(0,0,0)$ vanishes, we obtain
\begin{align*}
\frac{\partial^{3} K_{1,2}}{\partial u^{3}}(0,0,0)&=\frac{1}{g(0,0,0)}\frac{\partial^{3} k_{2}}{\partial u^{3}}(0,0,0)-\frac{3}{g^{2}(0,0,0)}\frac{\partial k_{2}}{\partial u}(0,0,0)\frac{\partial^{2} g}{\partial u^{2}}(0,0,0)\\
&=\frac{1}{\s^{4}k_{1}(0,0,0)}\left(\s^{2}\frac{\partial^{3} k_{2}}{\partial u^{3}}(0,0,0)+3\frac{\partial^{2} g}{\partial u^{2}}(0,0,0)\right)\,.
\end{align*}
We compute that
\[\frac{\partial^{3} k_{2}}{\partial u^{3}}(0,0,0)=9\s^{2}k_{1}(0,0,0)-3\mu_{10}(\mu_{4}-\s^{4})-3\mu_{8}(\mu_{4}\s^{2}-\mu_{6})-3\mu_{6}(\mu_{6}\s^{2}-\mu_{4}^{2})\]
and, by formula~(\ref{Dguu}), we obtain
\[\frac{\partial^{2} K_{1,2}}{\partial u\partial v}(0,0,0)=\frac{9\s^{4}k_{1}(0,0,0)-3(\mu_{4}+4\s^{4})k_{1}(0,0,0)}{\s^{4}k_{1}(0,0,0)}=\frac{-3(\mu_{4}+\s^{4})}{\s^{4}}\,.\]
Finally
\[\frac{\partial^{4}v}{\partial x^{4}}(0,\s^{2},\mu_{4})=\frac{1}{\s^{8}}\left(\frac{-2\mu_{4}}{\s^{2}}+\frac{-3(\mu_{4}+\s^{4})}{\s^{2}}-3\frac{\mu_{4}-\s^{4}}{\s^{2}}\right)=\frac{-8\mu_{4}}{\s^{10}}\,.\]
Likewise we obtain
\begin{multline*}
\frac{\partial^{4}w}{\partial x^{4}}(0,\s^{2},\mu_{4})=\frac{1}{\s^{10}k_{1}(0,0,0)}\Big(2\mu_{4}\frac{\partial k_{3}}{\partial u}(0,0,0)-3\s^{2}\frac{\partial^{2} k_{3}}{\partial u \partial v}(0,0,0)\\
+\frac{3}{k_{1}(0,0,0)}\frac{\partial k_{3}}{\partial u}(0,0,0)\frac{\partial g}{\partial v}(0,0,0)+\s^{2}\frac{\partial^{3} k_{3}}{\partial u^{3}}(0,0,0)\\
-\frac{3}{k_{1}(0,0,0)}\frac{\partial k_{3}}{\partial u}(0,0,0)\frac{\partial^{2} g}{\partial u^{2}}(0,0,0)\Big)\,.
\end{multline*}
We compute that
\[\frac{\partial k_{3}}{\partial u}(0,0,0)=\frac{\partial^{2} k_{3}}{\partial u\partial v}(0,0,0)=0\qquad \mbox{and} \qquad \frac{\partial^{2} k_{3}}{\partial u\partial v}(0,0,0)=2k_{1}(0,0,0)\,,\]
so that
\[\frac{\partial^{4}w}{\partial x^{4}}(0,\s^{2},\mu_{4})=\frac{2}{\s^{8}}\,.\]
Next we have
\[\frac{\partial K_{1,1}}{\partial w}(0,0,0)=\frac{1}{\s^{4}k_{1}(0,0,0)}\left(\s^{2}\frac{\partial k_{1}}{\partial w}(0,0,0)-\frac{\partial g}{\partial w}(0,0,0)\right)\]
and we compute that the partial derivative of $k_{1}$ with respect to $w$ taken at $(0,0,0)$ is equal to
\[-3\mu_{4}k_{1}(0,0,0)+2\mu_{10}(\mu_{4}\s^{2}-\mu_{6})+\mu_{12}(\mu_{4}-\s^{4})+\mu^{2}_{8}+\mu^{2}_{6}\mu_{4}-2\mu_{8}\mu_{4}^{2}\,,\]
and that the partial derivative of $g$ with respect to $w$ at $(0,0,0)$ is equal to
\begin{multline*}
(\mu_{6}-4\mu_{4}\s^{2})k_{1}(0,0,0)+2\mu_{10}\s^{2}(\mu_{4}\s^{2}-\mu_{6})+\mu_{12}\s^{2}(\mu_{4}-\s^{4})\\
+\s^{2}(\mu^{2}_{8}+\mu^{2}_{6}\mu_{4}-2\mu_{8}\mu_{4}^{2})\,.
\end{multline*}
Whence
\[\frac{\partial K_{1,1}}{\partial w}(0,0,0)=\frac{-3\s^{2}\mu_{4}k_{1}(0,0,0)-(\mu_{6}-4\mu_{4}\s^{2})k_{1}(0,0,0)}{\s^{4}k_{1}(0,0,0)}=\frac{\mu_{4}\s^{2}-\mu_{6}}{\s^{4}}\,.\]
We insert these previous results in the expression in the formula~(\ref{DIxxxxxx}) of the fourth partial derivative of $I_{\point }$ with respect to $x$ taken at $(0,\s^{2},\mu_{4})$ :
\begin{multline*}
\frac{\partial^{6}I_{\point }}{\partial x^{6}}(0,\s^{2},\mu_{4})=\frac{3}{\s^{8}}\frac{\partial^{2} K_{1,1}}{\partial v^{2}}(0,0,0)-\frac{6}{\s^{8}}\frac{\partial^{3} K_{1,1}}{\partial u^{2}\partial v}(0,0,0)+\frac{1}{\s^{8}}\frac{\partial^{4} K_{1,1}}{\partial u^{4}}(0,0,0)\\
+\frac{16\mu_{4}^{2}-2\mu_{6}\s^{2}+2\mu_{4}\s^{4}}{\s^{14}}\,.
\end{multline*}
We have
\[\frac{\partial^{2} K_{1,1}}{\partial v^{2}}=\frac{1}{g} \frac{\partial^{2}k_{1}}{\partial v^{2}}-\frac{2}{g^{2}}\frac{\partial g }{\partial v}\,\frac{\partial k_{1}}{\partial v}-\frac{k_{1}}{g^{2}}\frac{\partial^{2} g }{\partial v^{2}}+\frac{2k_{1}}{g^{3}}\left(\frac{\partial g }{\partial v}\right)^{2}\,.\]
Thus
\begin{multline*}
\frac{\partial^{2} K_{1,1}}{\partial v^{2}}(0,0,0)=\frac{1}{\s^{4}k_{1}(0,0,0)}\left( \s^{2}\frac{\partial^{2}k_{1}}{\partial v^{2}}(0,0,0)-\frac{\partial^{2} g }{\partial v^{2}}(0,0,0)\right)\\
-\frac{2}{\s^{6}k^{2}_{1}(0,0,0)}\left( \s^{2}\frac{\partial k_{1}}{\partial v}(0,0,0)-\frac{\partial g }{\partial v}(0,0,0) \right)\frac{\partial g }{\partial v}(0,0,0)\\
=\frac{1}{\s^{4}k_{1}(0,0,0)}\left( \s^{2}\frac{\partial^{2}k_{1}}{\partial v^{2}}-\frac{\partial^{2} g }{\partial v^{2}}-2\s^{2}\frac{\partial K_{1,1}}{\partial v}\frac{\partial g }{\partial v}\right)(0,0,0)\,.
\end{multline*}
We already know the values at $(0,0,0)$ of the partial derivatives of $g$ and $k_{1}$ with respect to $v$. Moreover, the properties of the function $f_{i}$, $i \in \N$, and their partial derivatives give us, after factorisation,
\begin{multline*}
\s^{2}\frac{\partial^{2}k_{1}}{\partial v^{2}}(0,0,0)-\frac{\partial^{2} g }{\partial v^{2}}(0,0,0)=(7\s^{2}\mu_{4}-\mu_{6}-8\s^{6})k_{1}(0,0,0)-2\mu_{10}(\mu_{4}-\s^{4})^{2}\\
+2\mu_{6}(\mu_{4}-\s^{4})(\mu_{8}+\mu_{4}^{2}-2\s^{2}\mu_{6})-2\s^{2}(\mu_{6}\s^{2}-\mu_{4}^{2})^{2}\,.
\end{multline*}
As a consequence
\begin{multline*}
\s^{4}k_{1}(0,0,0)\frac{\partial^{2} K_{1,1}}{\partial v^{2}}(0,0,0)=(7\s^{2}\mu_{4}-\mu_{6}-8\s^{6})k_{1}(0,0,0)-2\mu_{10}(\mu_{4}-\s^{4})^{2}\\
+2\mu_{6}(\mu_{4}-\s^{4})(\mu_{8}+\mu_{4}^{2}-2\s^{2}\mu_{6})-2\s^{2}(\mu_{6}\s^{2}-\mu_{4}^{2})^{2}-2(\s^{4}-\mu_{4})\Big(\mu_{10}(\mu_{4}-\s^{4})\\
+\mu_{8}(\mu_{4}\s^{2}-\mu_{6})+\mu_{6}(\mu_{6}\s^{2}-\mu_{4}^{2})\Big)-\frac{2(\s^{4}-\mu_{4})(\mu_{4}-4\s^{4})k_{1}(0,0,0)}{\s^{2}}\,.
\end{multline*}
Thus
\begin{multline*}
\s^{4}k_{1}(0,0,0)\frac{\partial^{2} K_{1,1}}{\partial v^{2}}(0,0,0)=2\s^{2}\left((\mu_{4}-\s^{4})(\mu_{4}\mu_{8}-\mu_{6}^{2})-(\mu_{6}\s^{2}-\mu_{4}^{2})^{2}\right)
\\+\frac{2\mu_{4}^{2}+\s^{4}\mu_{4}+\s^{2}\mu_{6}-4\s^{8}}{\s^{2}}k_{1}(0,0,0)\,.
\end{multline*}
By developing we get $(\mu_{4}-\s^{4})(\mu_{4}\mu_{8}-\mu_{6}^{2})-(\mu_{6}\s^{2}-\mu_{4}^{2})^{2}=\mu_{4}k_{1}(0,0,0)$. Thus
\[\frac{\partial^{2} K_{1,1}}{\partial v^{2}}(0,0,0)=\frac{2\mu_{4}^{2}-\s^{4}\mu_{4}-\s^{2}\mu_{6}}{\s^{6}}\,.\]
Next, since the partial derivatives of $g$ and $k_{1}$ with respect to $u$ are null at $(0,0,0)$, we get
\begin{multline*}
\frac{\partial^{3} K_{1,1}}{\partial u^{2}\partial v}(0,0,0)=\frac{1}{g(0,0,0)}\frac{\partial^{3}k_{1}}{\partial u^{2}\partial v}(0,0,0)-\frac{1}{g^{2}(0,0,0)}\frac{\partial^{2}k_{1}}{\partial u^{2}}(0,0,0)\frac{\partial g}{\partial v}(0,0,0)\\
-\frac{1}{g^{2}(0,0,0)}\frac{\partial k_{1}}{\partial v}(0,0,0)\frac{\partial^{2}g}{\partial u^{2}}(0,0,0)+\frac{2k_{1}(0,0,0)}{g^{3}(0,0,0)}\frac{\partial^{2}g}{\partial u^{2}}(0,0,0)\frac{\partial g}{\partial v}(0,0,0)\\
-\frac{k_{1}(0,0,0)}{g^{2}(0,0,0)}\frac{\partial^{3}g}{\partial u^{2}\partial v}(0,0,0)\,.
\end{multline*}
As in the computation of $\left(\partial^{2} K_{1,1}/\partial v^{2}\right)(0,0,0)$, we notice that this expression can be written as a function of the second partial derivative of $K_{1,1}$ with respect to $u$ and of the partial derivative of $K_{1,1}$ with respect to $v$ :
\begin{multline*}
\frac{\partial^{3} K_{1,1}}{\partial u^{2}\partial v}(0,0,0)=\frac{1}{\s^{4}k_{1}(0,0,0)}\Big(\s^{2}\frac{\partial^{3}k_{1}}{\partial u^{2}\partial v}(0,0,0)-\frac{\partial^{3}g}{\partial u^{2}\partial v}(0,0,0)\\
-\s^{2}\frac{\partial^{2}K_{1,1}}{\partial u^{2}}(0,0,0)\frac{\partial g}{\partial v}(0,0,0)-\s^{2}\frac{\partial K_{1,1}}{\partial v}(0,0,0)\frac{\partial^{2}g}{\partial u^{2}}(0,0,0)\Big)\,.
\end{multline*}
After factorising, this is equal to
\begin{multline*}
\frac{1}{\s^{4}k_{1}(0,0,0)}\Big(\s^{2}\frac{\partial^{3}k_{1}}{\partial u^{2}\partial v}(0,0,0)-\frac{\partial^{3}g}{\partial u^{2}\partial v}(0,0,0)-\frac{2(\mu_{4}^{2}-4\s^{8})}{\s^{2}}k_{1}(0,0,0)\\
-2\s^{4}\big(\mu_{10}(\mu_{4}-\s^{4})+\mu_{8}(\mu_{4}\s^{2}-\mu_{6})+\mu_{6}(\mu_{6}\s^{2}-\mu_{4}^{2})\big)\Big)\,.
\end{multline*}
and the properties of the functions $f_{i}$, $i \in \N$, and their partial derivatives, give us, after factorising by $k_{1}(0,0,0)$,
\begin{multline*}
\s^{2}\frac{\partial^{3}k_{1}}{\partial u^{2}\partial v}(0,0,0)-\frac{\partial^{3}g}{\partial u^{2}\partial v}(0,0,0)=(\s^{2}\mu_{4}+\mu_{6}-8\s^{6})k_{1}(0,0,0)\\
+2\s^{4}\big(\mu_{10}(\mu_{4}-\s^{4})+\mu_{8}(\mu_{4}\s^{2}-\mu_{6})+\mu_{6}(\mu_{6}\s^{2}-\mu_{4}^{2})\big)\,.
\end{multline*}
As a consequence
\begin{align*}
\frac{\partial^{3} K_{1,1}}{\partial u^{2}\partial v}(0,0,0)&=\frac{1}{\s^{4}k_{1}(0,0,0)}\Big(-\frac{2(\mu_{4}^{2}-4\s^{8})}{\s^{2}}+\s^{2}\mu_{4}+\mu_{6}-8\s^{6}\Big)k_{1}(0,0,0)\\
&=\frac{-2\mu_{4}^{2}+\s^{4}\mu_{4}+\s^{2}\mu_{6}}{\s^{6}}\,.
\end{align*}
We finish this proof by computing the fourth partial derivative of $K_{1,1}$ with respect to $u$ taken at $(0,0,0)$. Since the partial derivatives of $g$ and $k_{1}$ with respect to $u$ are null at $(0,0,0)$, we get
\begin{multline*}
\frac{\partial^{4} K_{1,1}}{\partial u^{4}}(0,0,0)=\frac{1}{g(0,0,0)}\frac{\partial^{4}k_{1}}{\partial u^{4}}(0,0,0)-\frac{k_{1}(0,0,0)}{g^{2}(0,0,0)}\frac{\partial^{4}g}{\partial u^{4}}(0,0,0)\\
-\frac{6}{g^{2}(0,0,0)}\frac{\partial^{2}k_{1}}{\partial u^{2}}(0,0,0)\frac{\partial^{2}g}{\partial u^{2}}(0,0,0)+\frac{6k_{1}(0,0,0)}{g^{3}(0,0,0)}\left(\frac{\partial^{2}g}{\partial u^{2}}(0,0,0)\right)^{2}\,.
\end{multline*}
After factorisation, it is equal to
\[\frac{1}{\s^{4}k_{1}(0,0,0)}\Big(\s^{2}\frac{\partial^{4}k_{1}}{\partial u^{4}}(0,0,0)-\frac{\partial^{4}g}{\partial u^{4}}(0,0,0)-6\s^{2}\frac{\partial^{2}K_{1,1}}{\partial u^{2}}(0,0,0)\frac{\partial^{2}g}{\partial u^{2}}(0,0,0)   \Big)\,.\]
We compute that
\begin{multline*}
\s^{2}\frac{\partial^{4}k_{1}}{\partial u^{4}}(0,0,0)-\frac{\partial^{4}g}{\partial u^{4}}(0,0,0)=-(23\s^{2}\mu_{4}+5\mu_{6}+24\s^{6})k_{1}(0,0,0)\\
+6(\mu_{4}+\s^{4})\big(\mu_{10}(\mu_{4}-\s^{4})+\mu_{8}(\mu_{4}\s^{2}-\mu_{6})+\mu_{6}(\mu_{6}\s^{2}-\mu_{4}^{2})\big)
\end{multline*}
and we have already computed that
\begin{multline*}
\frac{\partial^{2}g}{\partial u^{2}}(0,0,0)\frac{\partial^{2}K_{1,1}}{\partial u^{2}}(0,0,0)=\frac{\mu_{4}+\s^{4}}{\s^{4}}\Big(-(\mu_{4}+4\s^{4})k_{1}(0,0,0)\\
+\s^{2}\mu_{10}(\mu_{4}-\s^{4})+\s^{2}\mu_{8}(\mu_{4}\s^{2}-\mu_{6})+\s^{2}\mu_{6}(\mu_{6}\s^{2}-\mu_{4}^{2})\Big)\,.
\end{multline*}
Thus we get
\begin{multline*}
\s^{4}k_{1}(0,0,0)\frac{\partial^{4} K_{1,1}}{\partial u^{4}}(0,0,0)=-(23\s^{2}\mu_{4}+5\mu_{6}+24\s^{6})k_{1}(0,0,0)\\
+\frac{6(\mu_{4}+\s^{4})(\mu_{4}+4\s^{4})}{\s^{2}}-k_{1}(0,0,0)\,.
\end{multline*}
Hence
\[\frac{\partial^{4} K_{1,1}}{\partial u^{4}}(0,0,0)=\frac{7\s^{4}\mu_{4}-5\mu_{6}\s^{2}+6\mu_{4}^{2}}{\s^{6}}\,.\]
We have then
\begin{multline*}
\frac{\partial^{6}I_{\point }}{\partial x^{6}}(0,\s^{2},\mu_{4})=\frac{3(2\mu_{4}^{2}-\s^{4}\mu_{4}-\s^{2}\mu_{6})}{\s^{14}}-\frac{6(-2\mu_{4}^{2}+\s^{4}\mu_{4}+\s^{2}\mu_{6})}{\s^{14}}\\
+\frac{7\s^{4}\mu_{4}-5\mu_{6}\s^{2}+6\mu_{4}^{2}}{\s^{14}}+\frac{16\mu_{4}^{2}-2\mu_{6}\s^{2}+2\mu_{4}\s^{4}}{\s^{14}}\,.
\end{multline*}
\noindent We have shown that
\[\frac{\partial^{6}I_{\point }}{\partial x^{6}}(0,\s^{2},\mu_{4})=\frac{40\mu_{4}^{2}-16\s^{2}\mu_{6}}{\s^{14}}\,.\]
Thus
\[\frac{1}{\a!\b!\g!}\,\frac{\partial^{\a+\b+\g} I_{\point }}{\partial x^{\a}\partial y^{\b}\partial z^{\g}}(0,\s^{2},\mu_{4})=\left\{\begin{array}{cl}
\displaystyle{0} & \quad \mbox{if} \quad (\a,\b,\g)=(5,0,0)\,,\\
\displaystyle{\frac{5\mu_{4}^{2}-2\s^{2}\mu_{6}}{90\s^{14}}} & \quad \mbox{if} \quad (\a,\b,\g)=(6,0,0)\,.
\end{array}\right.\]

\subsubsection{Conclusion}

\noindent The term $A$ is then $(5\mu_{4}^{2}-2\s^{2}\mu_{6})/(90\s^{14})$. The computations of the previous section imply that

\begin{prop} Let $\r$ be a symmetric probability measure on $\R$ whose support contains at least five points. We suppose that $(0,0,0) \in \Dro_{\L_{\point }}$ and that
\[5\mu_{4}^{2}>2\s^{2}\mu_{6}.\]
Let $q$ denote the definite positive quadratic form on $\R^2$ given by
\[\forall (y,z) \in \R^2 \qquad q(y,z)=\frac{\mu_{8}-\mu_{4}^{2}}{2a}y^{2}+\frac{\mu_{4}\s^{2}-\mu_{6}}{a}yz+\frac{\mu_{4}-\s^{4}}{2a}z^{2}\,.\]
Then, in the neighbourhood of $(0,\s^2,\mu_4)$,
\[I_{\point }(x,y,z)-\frac{x^{2}}{2\s^{2}}+\frac{x^{2}(y-\s^2)}{2\s^4}-\frac{\mu_4 x^4}{12 \s^8}\sim \frac{(5\mu_{4}^{2}-2\s^{2}\mu_{6})x^6}{90\s^{14}}+q(y-\s^{2},z-\mu_{4})\,.\]
\label{DevI*}
\end{prop}

\noindent For many usual distributions, the term $5\mu_{4}^{2}-2\s^{2}\mu_{6}$ is positive. For example $5\mu_{4}^{2}-2\s^{2}\mu_{6}=12b^2c^8>0$, for $\r=(1-2b)\d_0+b\d_{-c}+b\d_c$ with $c>0$ and $b\in \,]0,1/2[$. However we can find a probability measure on $\R$ for which this term is non-positive. To this end, it is enough to take a measure whose sixth moment explodes compared to the fourth moment. Let us consider the measure with density
\[x\longmapsto\frac{1}{1+x^6}\ind{[-5,5]}(x)\,\left(\int_{-5}^{5}\frac{dy}{1+y^6}\right)^{-1}\]
with respect to the Lebesgue measure on $\R$. Its moments can be computed in simple fractions $X^{\a}/(1+X^6)$, $\a\in \{2,4,6\}$. We compute that $5\mu_{4}^{2}-2\s^{2}\mu_{6}$ is non-positive (an approaching value is $-0.483$).

\section{Construction of an interaction term}
\label{ConstructionInteraction}

\noindent In this section we investigate how to build an interaction term whose associated model is amenable to mathematical analysis. We first find criteria on $\r$ and $H$ so that $I_{\point}-H$ has a unique minimum at $(0,\s^2,\mu_4)$ in any compact subset of~$\Theta^{\!*}$ whose interior contains $(0,\s^2,\mu_4)$, and so that the expansion of $I_{\point}-F-R$ still holds for $I_{\point}-H$. Next we extend the criteria on $H$ in order to control what happens outside any compact of~$\Theta^{\!*}$, expecially what happens around the line $x=y=0$ of $\R^3$. We use then a variant of Varadhan's lemma. We end this section by proving that the function $H$ given in the introduction satisfies these criteria.

\subsection{First investigations}
\label{FirstInvestigations}

\noindent Let us suppose that $\r$ is a symmetric probability measure on $\R$ whose support contains at least five points. We assume that
 \[(0,0,0) \in \Dro_{\L_{\point }} \qquad \mbox{and} \qquad 5\mu_{4}^{2}>2\s^{2}\mu_{6}\,.\]

\noindent In section~\ref{Preliminaries} we saw that it seems natural to consider the interacting function $H=F+R$ with
\[R: (x,y,z)\in \R\times \R\backslash\{0\}\times\R\longmapsto \frac{zx^{4}}{12y^{4}}\,.\]
It satisfies
\[\forall n \geq 1\quad \forall (x,y,z)\in \R\times \R\backslash\{0\}\times\R \qquad \frac{1}{n}H(xn,yn,zn)=H(x,y,z)\,.\]

\noindent In the neighbourhood of $(0,\s^{2},\mu_{4})$,
\[F(x,y)+R(x,y,z)=\frac{x^{2}}{2\s^{2}}\frac{1}{1+h}+\frac{\mu_{4}x^{4}}{12\s^{8}}\frac{1}{(1+h)^{4}}+\frac{(z-\mu_{4})x^{4}}{12\s^{8}}\frac{1}{(1+h)^{4}}\,,\]
where $h=(y-\s^{2})/\s^{2}$. In the neighbourhood of $0$, we have
\[\frac{1}{1+h}=1-h+h^{2}-h^{3}+h^{4}+o(h^{4})\,,\]
\[\frac{1}{(1+h)^{4}}=1-4h+10h^{2}-20h^{3}+o(h^{3})\,.\]
Thus, in the neighbourhood of $(0,\s^{2},\mu_{4})$,
\begin{multline*}
F(x,y)+R(x,y,z)=\frac{x^{2}}{2\s^{2}}-\frac{x^{2}(y-\s^{2})}{2\s^{4}}+\frac{x^{2}(y-\s^{2})^{2}}{2\s^{6}}+\frac{\mu_{4}x^{4}}{12\s^{8}}\\-\frac{x^{2}(y-\s^{2})^{3}}{2\s^{8}}-\frac{\mu_{4}x^{4}(y-\s^{2})}{3\s^{10}}+\frac{x^{4}(z-\mu_{4})}{12\s^{8}}+\frac{x^{2}(y-\s^{2})^{4}}{2\s^{10}}\\+\frac{5\mu_{4}x^{4}(y-\s^{2})^{2}}{6\s^{12}}-\frac{x^{4}(y-\s^{2})(z-\mu_{4})}{3\s^{10}}+o(\|x,y-\s^{2},z-\mu_{4}\|^{6})\,.
\end{multline*}

\noindent Lemma~\ref{negligeable} implies that, in the neighbourhood of $(0,\s^{2},\mu_{4})$,
\begin{multline*}
F(x,y)+R(x,y,z)=\frac{x^{2}}{2\s^{2}}-\frac{x^{2}(y-\s^{2})}{2\s^{4}}+\frac{\mu_{4}x^{4}}{12\s^{8}}\\+o\left(\frac{(5\mu_{4}^{2}-2\s^{2}\mu_{6})x^6}{90\s^{14}}+q(y-\s^{2},z-\mu_{4})\right)
\end{multline*}
and then it follows from proposition~\ref{DevI*} that, in the neighbourhood of $(0,\s^{2},\mu_{4})$,
\[I_{\point }(x,y,z)-F(x,y)-R(x,y,z) \sim \frac{(5\mu_{4}^{2}-2\s^{2}\mu_{6})x^6}{90\s^{14}}+q(y-\s^{2},z-\mu_{4})\,.\]
\label{eq-x6}

\noindent The computations of the previous section show that, in the neighbourhood of $0$,
\[I_{\point }(x,\s^{2},\mu_{4})-F(x,\s^{2})-R(x,\s^{2},\mu_{4}) \sim \frac{5\mu_{4}^{2}-2\s^{2}\mu_{6}}{90\s^{14}}x^{6}\,.\]
As a consequence, if $5\mu_{4}^{2}<2\s^{2}\mu_{6}$ then $I_{\point }-F-R$ does not have a minimum at $(0,\s^{2},\mu_{4})$. Thus it is not possible to prove that $I_{\point }-F-R$ is non-negative for any symmetric probability measures on $\R$, as we did in~\cite{CerfGorny} for $I-F$. The techniques we used there have not been successful and we have not been able to show that $I_{\point }-F-R$ has a unique minimum for an interesting class of probability measures on $\R$. We will go around this problem by modifying the interacting function $H$ in order to \og force \fg{} the function
\[G_n : (x,y,z) \longmapsto I_{\point }(x,y,z)-\frac{1}{n}H(nx,ny,nz)\]
to have a unique minimum at $(0,\s^2,\mu_4)$ for any $n\geq 1$, and to have the same expansion we obtained above. \medskip

\noindent By analysing the essential ingredient of the proof of theorem~2~of~\cite{CerfGorny}, we consider the following hypothesis :

\begin{hyps} Let $\r$ be a symmetric probability measure on $\R$ whose support contains at least five points. We assume that
 \[(0,0,0) \in \Dro_{\L_{\point }} \qquad \mbox{and} \qquad 5\mu_{4}^{2}>2\s^{2}\mu_{6}\,.\]
 
\noindent Let $H$ be a function from $\Theta^{\!*}$ to $\R$. We suppose that there exists $(R_{n})_{n\geq 1}$ a sequence of upper semi-continuous functions from $\Theta^{\!*}$ to $\R$ satisfying, for any $(x,y,z) \in \Theta^{\!*}$,
\[\forall n\geq 1\qquad 0\leq R_{n+1}(x,y,z) \leq R_{n}(x,y,z)\leq R(x,y,z)\,,\]
\[\forall n\geq 1\qquad H(x,y,z)-F(x,y)=nR_{n}\left(\frac{x}{n},\frac{y}{n},\frac{z}{n}\right)\,,\]
\[R_{n}(x,y,z) \underset{n \to + \infty}{\longrightarrow}0\]
and, for every $(x,y,z) \in \R^{3}$,
\[n(R-R_{n})\left(\frac{x}{n^{1/6}},\frac{y}{\sqrt{n}}+\s^{2},\frac{z}{\sqrt{n}}+\mu_{4}\right) \underset{n \to + \infty}{\longrightarrow}0\,.\]
\label{hypHrho}
\end{hyps}
\vspace*{-0.5cm}

\noindent We have the following proposition :

\begin{prop} Suppose that $\r$ and $H$ satisfy the hypothesis~\ref{hypHrho}. If $q$ denotes the definite positive quadratic form of proposition~\ref{DevI*}, then, for any $(x,y,z) \in \R^{3}$,
\[n(I_{\point }-F-R_{n})\left(\frac{x}{n^{1/6}},\frac{y}{\sqrt{n}}+\s^{2},\frac{z}{\sqrt{n}}+\mu_{4}\right) \underset{n \to + \infty}{\longrightarrow}q(y,z)+\frac{5\mu_{4}^{2}-2\s^{2}\mu_{6}}{90\s^{14}}x^{6}\,.\]

\noindent Let $K$ be a compact subset of $\R^{3}$ included in $\Theta^{\!*}$ such that $(0,\s^{2},\mu_{4})$ belongs to the interior of $K$. There exists $n_{0}\geq 1$ such that $I_{\point }-F-R_{n_{0}}$ has a unique minimum on $K$ at $(0,\s^{2},\mu_{4})$.
\label{MinEq(I*-Fn)}
\end{prop}

\noindent We will use the following lemma, which is a variant of Dini's theorem :

\begin{lem} Let $(g_n)_{n\geq 1}$ be a non-increasing sequence of functions defined on a compact set $X$ and which converges pointwise to a function $g$ defined on $X$. If the function $g_n-g$ is upper semi-continuous for any $n\geq 1$, then $(g_n)_{n\geq 1}$ converges uniformly over $X$ towards $g$.
\end{lem}

\begin{proof} For any $n\geq 1$, we put $h_n=g_n-g$. The sequence $(h_n)_{n\geq 1}$ is non-increasing and converges pointwise to the null function. For a fixed $\eps>0$ and for any $n\geq 1$, we denote
\[A_n(\eps)=\{\,x\in X:h_n(x)<\eps\,\}\,.\]
These sets are open since, for any $n\geq 1$, the function $h_n$ is upper semi-continuous. The convergence of the sequence $(h_n)_{n\geq 1}$ implies that
\[X \subset \bigcup_{n\geq 1}A_n(\eps)\,.\]
We can extract a finite subcover : there exists $N\geq 1$ such that
\[X \subset \bigcup_{n\leq N}A_n(\eps)\,.\]
Since $(h_n)_{n\in \N}$ is non-decreasing, then $X\subset A_N(\eps)$. Thus
\[\forall x\in X \qquad \exists N>0 \qquad n\geq N \, \Longrightarrow \, h_n(x)\leq h_N(x)<\eps\,.\]
This proves the lemma.
\end{proof}
\medskip

\noindent{\bf Proof of proposition~\ref{MinEq(I*-Fn)}.}
In the neighbourhood of $(0,\s^{2},\mu_{4})$,
\[I_{\point }(x,y,z)-F(x,y)-R(x,y,z) \sim \frac{(5\mu_{4}^{2}-2\s^{2}\mu_{6})x^6}{90\s^{14}}+q(y-\s^{2},z-\mu_{4})\,.\]
For every $n\geq 1$, we denote \[G_{n}=I_{\point }-F-R_{n}=(I_{\point }-F-R)+(R-R_{n})\,.\]
The expansion of $I_{\point }-F-R_{n}$ and the hypothesis~\ref{hypHrho} imply that, for $(x,y,z) \in \R^3$,
\[n(I_{\point }-F-R_{n})\left(\frac{x}{n^{1/6}},\frac{y}{\sqrt{n}}+\s^{2},\frac{z}{\sqrt{n}}+\mu_{4}\right) \underset{n \to + \infty}{\longrightarrow}q(y,z)+\frac{5\mu_{4}^{2}-2\s^{2}\mu_{6}}{90\s^{14}}x^{6}\,.\]

\noindent Next the function $R-R_{n}$ is non-negative, thus $G_n\geq I_{\point }-F-R$ and there exists an open set $U$ centered at $(0,\s^{2},\mu_{4})$ such that, for any $(x,y,z) \in U$,
\[G_{n}(x,y,z)\geq \frac{1}{2}q(y-\s^{2},z-\mu_{4})+ \frac{5\mu_{4}^{2}-2\s^{2}\mu_{6}}{180\s^{14}}x^{6}\,.\]
The right term of this inequality is non-negative since $5\mu_{4}^{2}>2\s^{2}\mu_{6}$. Since $q$ is a definite positive quadratic form, this term vanishes only at $(0,\s^{2},\mu_{4})$. Thus we proved that, for any $n \geq 1$, $G_{n}$ has a unique minimum on $U$ at $(0,\s^{2},\mu_{4})$ and it is equal to $0$.\medskip

\noindent Without loss of generality, we can suppose that $U$ is included in $K$. The set $K\cap U^{c}$ is a compact subset of $\R^{3}$ included in $\Theta^{\!*}$. Let $\nu_{\r}$ be the law of $(Z,Z^2)$ when $Z$ is a random variable with distribution $\r$. We denote by $\L$ the Log-Laplace of $\nu_{\r}$ and by $I$ its Cram\'er transform. The measure $\r$ is symmetric and $(0,0,0)\in \Dro_{\L_{\point }}$. Moreover we have
\[\forall (u,v) \in \R^2 \qquad \L(u,v)=\L_{\point }(u,v,0)\,.\]
As a consequence $(0,0)\in \Dro_{\L}$ and proposition~\ref{minI-F} implies that the function $I-F$ has a unique minimum at $(0,\s^{2})$ on $\R\times\R\backslash{\{0\}}$. Next, for any $(x,y,z,u,v)\in \R^5$,
\[I_{\point }(x,y,z)\geq  xu+yv+z\times 0-\L_{\point }(u,v,0)=xu+yv-\L(u,v)\,.\]
Taking the supremum over $(u,v)\in \R^2$, it comes that
\[\forall (x,y,z) \in \R\times\R\backslash{\{0\}}\times \R \qquad I_{\point }(x,y,z)-\frac{x^2}{2y}\geq I(x,y)-\frac{x^2}{2y}\,.\]

\noindent Hence, for $(x,y,z) \in K\cap U^{c}$, there are two cases :\smallskip

\noindent $\star$ Either $(x,y)\neq (0,\s^2)$ and then
$I_{\point }(x,y,z)-F(x,y)>0$.\smallskip

\noindent $\star$ Or $(x,y)=(0,\s^2)$ and then $z\neq \mu_4$. The function $I_{\point }$ has a unique minimum at $(0,\s^2,\mu_4)$ in which it is null (see chapter V of~\cite{Rockafellar} for a proof of this result). Thus
\[I_{\point }(0,\s^2,z)-F(0,\s^2)=I_{\point }(0,\s^2,z)>0\,.\]

\noindent In each case
\[\forall (x,y,z) \in K\cap U^{c} \qquad I_{\point }(x,y,z)-\frac{x^2}{2y}>0\,.\]
By hypothesis, the sequence of functions $(R_{n}+F-I_{\point })_{n \geq 1}$ is non-increasing and converges pointwise to $F-I_{\point }$. Moreover, for any $n\geq 1$, $R_{n}+F-I_{\point }$ is upper semi-continuous. Hence the previous lemma implies that $(I_{\point }-F-R_{n})_{n \geq 1}$ converges uniformly to $I_{\point }-F$ on $K\cap U^{c}$. As a consequence there exists $n_0\geq 1$ such that $I_{\point }-F-R_{n_{0}}$ is positive on $K\cap U^{c}$. Hence $I_{\point }-F-R_{n_{0}}$ has a unique minimum on $K\cap U^{c}$ at $(0,\s^{2},\mu_{4})$.\qed

\subsection{Around Varadhan's lemma}

\noindent We saw in section~\ref{Preliminaries} that the law of $(S_{n}/n,T_{n}/n,U_{n}/n)$ under $\widetilde{\mu}_{H,n,\r}$ is
\[\frac{\displaystyle{e^{nF(x,y)+nR_n(x,y,z)}\ind{\Theta^{\!*}}(x,y,z)\,d\widetilde{\nu}_{\point n,\r}(x,y,z)}}{\displaystyle{\int_{\Theta^{\!*}}e^{nF(x,y)+nR_n(x,y,z)}\,d\widetilde{\nu}_{\point n,\r}(x,y,z)}}\,.\]

\noindent We search additional conditions on $H$ and $\r$ so that, if $A$ is a closed set which does not contain $(0,\s^{2},\mu_4)$, then  
\[\limsupn \frac{1}{n}\ln \int_{\Theta^{\!*}\cap A}e^{nF(x,y)+nR_n(x,y,z)}\,d\widetilde{\nu}_{\point n,\r}(x,y,z)<0\,.\]

\noindent To this end, we need a variant of Varadhan's lemma. By proposition~\ref{MinEq(I*-Fn)} we can conclude if $A$ is a compact subset of $\Theta^{\!*}$. We have to extend the criteria on $H$ in order to control what happens around the line $x=y=0$ of $\R^3$. We proceed similarly as in~\cite{CerfGorny}.\medskip

\begin{hyps} Assume that $\r$ and $H$ satisfy the hypothesis~\ref{hypHrho}. We suppose that $\r$ has a bounded support and that, for any $r>0$, there exists $\d>0$ such that
\[\forall (x,y,z) \in \Theta \cap(\R\times]0,\d] \times \R) \quad \forall n \geq 1 \qquad  R_n(x,y,z)\leq r\,.\]
\label{hypHrho2}
\end{hyps}

\begin{hyps} Assume that $\r$ and $H$ satisfy the hypothesis~\ref{hypHrho}. We suppose that there exists $c_0>0$ such that
\[\forall (x,y,z) \in \Theta^{\!*} \quad \forall n\geq 1 \qquad R_n(x,y,z)\leq c_0 y\,.\]
\label{hypHrho2bis}
\end{hyps}

\begin{prop} Let $\r$ and $H$ fulfill the hypothesis~\ref{hypHrho}. We have
\[\liminfn \frac{1}{n}\ln \int_{\Theta^{\!*}}e^{nF(x,y)+nR_n(x,y,z)}\,d\widetilde{\nu}_{\point n,\r}(x,y,z)\geq 0\,.\]
Suppose that $\r$ and $H$ also satisfy either the hypothesis~\ref{hypHrho2} or the hypothesis~\ref{hypHrho2bis}. Then,
for any closed subset $A$ of $\R^3$ which does not contain $(0,\s^{2},\mu_4)$, we have 
\[\limsupn \frac{1}{n}\ln \int_{\Theta^{\!*}\cap A}e^{nF(x,y)+nR_n(x,y,z)}\,d\widetilde{\nu}_{\point n,\r}(x,y,z)<0\,.\]
\label{TypeVaradhan-x4}
\end{prop}

\begin{proof} By Jensen's inequality, we have
\begin{multline*}
\liminfn \frac{1}{n}\ln \int_{\Theta^{\!*}}e^{nF(x,y)+nR_n(x,y,z)}\,d\widetilde{\nu}_{\point n,\r}(x,y,z)\\\geq \liminfn \int_{\Theta^{\!*}}(F(x,y)+R_n(x,y,z))\,d\widetilde{\nu}_{\point n,\r}(x,y,z)\geq 0.
\end{multline*}
Let us show the second inequality. Proposition~4~of~\cite{Gorny3} states that there exists $\g>0$ such that, for $\d \in \,]0,\s^2[$ small enough and $n$ large enough,
\[ \int_{\R^2}e^{nF(x,y)}\ind{x^2\leq y}\ind{0<y\leq \d}\,d\widetilde{\nu}_{n,\r}(x,y) \leq e^{-n\g}\,,\]
where $\widetilde{\nu}_{n,\r}$ denotes the law of $(S_n/n,T_n/n)$ under $\r^{\otimes n}$.\medskip

\noindent The function $H$ satisfies the hypothesis~\ref{hypHrho2} or~\ref{hypHrho2bis} thus we can choose $\d$ small enough so that
\[\forall (x,y,z) \in \Theta \cap(\R\times]0,\d] \times \R) \quad \forall n \geq 1 \qquad  R_n(x,y,z)\leq \frac{\gamma}{2}\,.\]
Hence
\begin{align*}
\int_{\Theta^{\!*}}e^{nF(x,y)+nR_n(x,y,z)}\ind{y\leq \d}&\,d\widetilde{\nu}_{\point n,\r}(x,y,z)\\
&\leq e^{n\g/2}\int_{\Theta^{\!*}}e^{nF(x,y)}\ind{y\leq \d}\,d\widetilde{\nu}_{\point n,\r}(x,y,z)\\
&\leq e^{n\g/2}\int_{\R^3}e^{nF(x,y)}\ind{x^2\leq y}\ind{0<y\leq \d}\,d\widetilde{\nu}_{\point n,\r}(x,y,z)\\
&=e^{n\g/2} \int_{\R^2}e^{nF(x,y)}\ind{x^2\leq y}\ind{0<y\leq \d}\,d\widetilde{\nu}_{n,\r}(x,y)\\
&\leq e^{n\g/2}e^{-n\g}=e^{-n\g/2}\,.
\end{align*}
Thus, for $\d$ small enough,
\[\limsupn \frac{1}{n}\ln \int_{\Theta^{\!*}}e^{nF(x,y)+nR_n(x,y,z)}\ind{y\leq \d}\,d\widetilde{\nu}_{\point n,\r}(x,y,z) \leq -\g/2\,.\]
We define $A_{\d}=\{\,(x,y,z)\in \Theta \cap A : y\geq \d\,\}$. We have
\[\Theta^{\!*}\cap A \subset \{\,(x,y,z)\in \Theta^{\!*} : y\leq \d\,\} \cup A_{\d}\,.\]

\noindent If $\r$ and $H$ satisfy the hypothesis~\ref{hypHrho2bis}, we have, for any $(x,y,z) \in \Theta^{\!*}$ and $n\geq 1$,
\[I_{\point }(x,y,z)-F(x,y)-R_n(x,y,z)\geq I_{\point }(x,y,z)-\frac{1}{2}-c_0 y \geq I_{\point }(x,y,z)-\frac{1}{2}-c_0 \sqrt{z}\,.\]
Since $(0,0,0)\in\Dro_{\L_{\point }}$, there exists $w_0>0$ small enough so that $(0,0,w_0) \in D_{\L_{\point }}$. By the definition of the Cram\'er transform, we have
\[I_{\point }(x,y,z)\geq 0\times x+0\times y+w_0 \times z-\ln \int_{\R}e^{0\times t+0\times t^2+w_0\times t^4}\,d\r(t)\,.\]
As a consequence
\[I_{\point }(x,y,z)-F(x,y)-R_n(x,y,z)\geq w_0 z -c_0\sqrt{z} -\frac{1}{2}-\ln\int_{\R}e^{w_0 t^4}\,d\r(t)\,.\]
The right term converges to $+\infty$ when $z$ goes to $+\infty$ and it does not depend on $x$, $y$ and $n$. As a consequence, there exists $z_0>0$ such that, for any $n \geq 1$,
\[\forall (x,y,z) \in \Theta^{\!*} \cap (\R\times\R\times[z_0,+\infty[) \qquad I_{\point }(x,y,z)-F(x,y)-R_n(x,y,z) \geq 1\,.\]
We put $K=\{\,(x,y,z) \in \Theta : z\leq \max(\,z_0,2\mu_4\,)\, \}$. The above inequality implies that
\[\inf_{n\geq 1}\,\inf_{A_{\d} \cap K^c} (\,I_{\point }-F-R_{n}\,)\geq 1\,.\]
Moreover, we can reduce $\d$ so that the set $\{\,(x,y,z)\in \Theta : y\geq \d\,\}\cap K$ is a compact subset of $\R^3$ included in $\Theta^{\!*}$ and whose interior contains $(0,\s^{2},\mu_{4})$. Thus proposition~\ref{MinEq(I*-Fn)} ensures the existence of $n_{0}\geq 1$ such that $I_{\point }-F-R_{n_{0}}$ has a unique minimum in $\{\,(x,y,z)\in \Theta : y\geq \d\,\}\cap K$ at $(0,\s^{2},\mu_{4})$. Since $I-F-R_{n_0}$ is a good rate function and $A_{\d}\cap K$ does not contain $(0,\s^2,\mu_4)$, we have
\[\inf_{A_{\d}\cap K} \,(\,I_{\point }-F-R_{n_{0}}\,)>0\,.\]
As a consequence
\[\inf_{A_{\d}} \,(\,I_{\point }-F-R_{n_{0}}\,)>0\,.\]

\noindent If $\r$ and $H$ satisfy the hypothesis~\ref{hypHrho2} and if $K'$ denotes the closed convex hull of the support of $\nu_{\point \r}$ (which is then compact), then we can also reduce $\d$ in order to apply proposition~\ref{MinEq(I*-Fn)} and find some $n_0\geq 1$ such that
\[\inf_{A_{\d}} \,(\,I_{\point }-F-R_{n_{0}}\,)=\inf_{A_{\d}\cap K'} \,(\,I_{\point }-F-R_{n_{0}}\,)>0\,.\]

\noindent In both cases, the usual Varadhan lemma (see~\cite{DZ}) implies that there exists $\g_1>0$ such that, for $n$ large enough,
\[ \int_{A_{\d}}e^{nF(x,y)+nR_{n_0}(x,y,z)}\,d\widetilde{\nu}_{\point n,\r}(x,y,z) \leq e^{-n\g_1 }\,.\]
Finally, since $R_n \leq R_{n_0}$ for any $n\geq n_0$, we have
\[\limsupn \frac{1}{n}\ln\int_{A_{\d}}e^{nF(x,y)+nR_{n}(x,y,z)}\,d\widetilde{\nu}_{\point n,\r}(x,y,z) \leq -\g_1\,.\]
Hence
\[\limsupn \frac{1}{n}\ln\int_{\Theta^{\!*}\cap A}e^{nF(x,y)+nR_{n}(x,y,z)}\,d\widetilde{\nu}_{\point n,\r}(x,y,z) \leq \max\left(-\frac{\g}{2}\,,\, -\g_1\right)<0\,.\]
This ends the proof of the proposition.\end{proof}

\subsection{One good candidate}
\label{candidate}

\noindent Let $\r$ satisfies the hypothesis~\ref{hypHrho}. One good candidate for $H$ is
\[H:(x,y,z) \in \Theta^{\!*} \longmapsto \frac{x^{2}}{2y}+\frac{1}{12}\frac{zx^{4}y^{5}}{y^{9}+x^{10}+zx^4y^4}\,.\]
Indeed, the sequence $(R_{n})_{n\geq 1}$ defined by
\[\forall (x,y,z)\in\Theta^{\!*} \quad \forall n \geq 1 \quad R_{n}(x,y,z)=\frac{1}{12}\frac{zx^{4}y^{5}}{y^{9}+nx^{10}+zx^4y^4}\,,\]
consists of upper semi-continuous functions and, for any $(x,y,z) \in \Theta^{\!*}$,
\[\forall n\geq 1\qquad 0\leq R_{n+1}(x,y,z) \leq R_{n}(x,y,z)\leq R(x,y,z)\,,\]
\[\forall n\geq 1\qquad H(x,y,z)-\frac{x^2}{2y}=nR_{n}\left(\frac{x}{n},\frac{y}{n},\frac{z}{n}\right)\,,\]
\[R_{n}(x,y,z) \underset{n \to + \infty}{\longrightarrow}0\,.\]
\noindent We have next
\begin{align*}
(R-R_n)(x,y,z)&=\frac{zx^{4}}{12y^{4}}-\frac{1}{12}\frac{zx^{4}y^{5}}{y^{9}+nx^{10}+zx^4y^4}=\frac{zx^{4}(nx^{10}+zx^4y^4)}{12y^{4}(y^{9}+nx^{10}+zx^4y^4)}\,.
\end{align*}
Evaluating in $(x/n^{1/6},y/\sqrt{n}+\s^{2},z/\sqrt{n}+\mu_{4})$, we get
\[n\left(\frac{x}{n^{1/6}}\right)^{10}+\left(\frac{z}{\sqrt{n}}+\mu_4\right)\left(\frac{x}{n^{1/6}}\right)^{4}\left(\frac{y}{\sqrt{n}}+\s^{2}\right)^{4}\sim \frac{x^{10}+\s^{8}\mu_4}{n^{2/3}}\,,\]
\[\left(\frac{y}{\sqrt{n}}+\s^{2}\right)^{9}+n\left(\frac{x}{n^{1/6}}\right)^{10}+\left(\frac{z}{\sqrt{n}}+\mu_4\right)\left(\frac{x}{n^{1/6}}\right)^{4}\left(\frac{y}{\sqrt{n}}+\s^{2}\right)^{4}\sim \s^{18}\]
and
\[\frac{(z/\sqrt{n}+\mu_{4})x^4n^{-2/3}}{12(y/\sqrt{n}+\s^{2})^4}\sim\frac{\mu_4 x^4}{12\s^8 n^{2/3}}\,.\]
Thus
\begin{align*}
n(R-R_n)\left(\frac{x}{n^{1/6}},\frac{y}{\sqrt{n}}+\s^{2},\frac{z}{\sqrt{n}}+\mu_{4}\right)&\sim n\frac{\mu_4 x^4}{12\s^{8} n^{2/3}\s^{18}}\frac{x^{10}+\s^{8}\mu_4}{n^{2/3}}\\
&\sim \frac{\mu_4 x^4(x^{10}+\s^{8}\mu_4)}{12 \s^{26} n^{1/3}}\underset{n \to + \infty}{\longrightarrow}0\,.
\end{align*}
Hence $H$ satisfies the hypothesis~\ref{hypHrho}. Finally, for any $(x,y,z)\in \Theta^{\!*}$, we have
\[zx^{4}y^{5}\leq y^{10}+nyx^{10}+zx^4y^5=y(y^{9}+nx^{10}+zx^4y^4)\,.\]
Thus $H$ also satisfies the hypothesis~\ref{hypHrho2bis}. 

\section{Fluctuations theorem}
\label{FluctuationsTheorem}

\noindent In this section, we suppose that $\r$ has a density $f$ with respect to the Lebesgue measure on $\R$. We will proceed as in section~7~of~\cite{CerfGorny} to obtain our fluctuations result: we first compute an asymptotic expression of the density of $\nu_{\point \r}^{*n}$, for $n$ large enough. Next we prove a generalisation of theorem~\ref{MainTheo} with the help of Laplace's method.

\subsection{Asymptotic expression of the density of $\nu_{\point \r}^{*n}$}
\label{calculdensite}

\begin{prop} If $\r$ is a probability measure having a density $f$ with respect to the Lebesgue measure on $\R$, then $\nu_{\point \r}^{*3}$ admits a density with respect to the Lebesgue measure on $\R^3$. Suppose that, for some $p \in  \, ]1,2]$,
\begin{equation}
\int_{\R^3}\frac{f^p(x)f^p(y)f^p(z)}{\left|(x+y+z)(x-y)(y-z)(z-x)\right|^{p-1}}\,dx\,dy\,dz<+\infty.\tag{$*$}
\end{equation}
Then, for $n$ large enough, $\widetilde{\nu}_{\point n,\r}$ has a density $g_{n}$ with respect to the Lebesgue measure on $\R^{3}$ such that, for any compact subset $K$ of $A_{I_{\point }}$, when $n$ goes to $+\infty$, uniformly over $(x,y,z)\in K$,
\[g_{n}(x,y,z)\sim \left(\frac{n}{2\pi}\right)^{3/2}\left(\mathrm{det}\, \mathrm{D}_{(x,y,z)}^{3}I_{\point }\right)^{1/2} e^{-nI_{\point }(x,y,z)}\,.\]
\label{densitef3}
\end{prop}

\begin{proof} Let $\Phi$ be a measurable positive function on $\R^3$. We have
\begin{align*}
\int_{\R^3}&\Phi(x,y,z)\,d\nu_{\point \r}^{*3}(x,y,z)\\
&=\int_{\R^9}\Phi(u_1+u_2+u_3,v_1+v_2+v_3,w_1+w_2+w_3)\\
&\qquad\qquad\times d\nu_{\point \r}(u_1,v_1,w_1)\,d\nu_{\point \r}(u_2,v_2,w_2)\,d\nu_{\point \r}(u_3,v_3,w_3)\\
&=\int_{\R^3}\Phi(x+y+z,x^2+y^2+z^2,x^4+y^4+z^4)\,d\r(x)\,d\r(y)\,d\r(z)\\
&=\int_{\R^3}\Phi(x+y+z,x^2+y^2+z^2,x^4+y^4+z^4)\,f(x)f(y)f(z)\,dx\,dy\,dz\,.
\end{align*}
Let us make the change of variables given by
\[\phi : (x,y,z) \in \R^3 \longmapsto (x+y+z,x^2+y^2+z^2,x^4+y^4+z^4)\,.\]
The function $\phi$ is $\Ck{1}$ on $\R^3$. We compute its Jacobian : for any $(x,y,z)\in \R^3 $, we have
\begin{align*}
\mathrm{Jac}_{(x,y,z)}\phi&=\begin{vmatrix}
1&1&1\\
2x&2y&2z\\
4x^3&4y^3&4z^3
\end{vmatrix}=8\begin{vmatrix}
1&0&0\\
x&y-x&z-x\\
x^3&y^3-x^3&z^3-x^3
\end{vmatrix}\\
&=8\left((y-x)(z^3-x^3)-(z-x)(y^3-x^3)\right)\\
&=8\left((y-x)(z-x)(z^2+xz+x^2)-(z-x)(y-x)(y^2+xy+x^2)\right)\\
&=8(y-x)(z-x)\left(z^2+xz+x^2-y^2-xy-x^2)\right)\\
&=8(y-x)(z-x)(z-y)(x+y+z)\,.
\end{align*}
We introduce the set
\begin{multline*}
\Hc=\{\,(x,y,z)\in \R^3 : x+y+z=0\,\}\,\cup\, \{\,(x,y,z)\in \R^3 : x=y\,\}\\\,\cup\, \{\,(x,y,z)\in \R^3 : x=z\,\}\,\cup\, \{\,(x,y,z)\in \R^3 : y=z\,\}\,.
\end{multline*}
It is the union of four hyperplanes on which the Jacobian of $\phi$ vanishes. We define next
\begin{align*}
O_{1}=\{\,(x,y,z)\in \R^3 : x<y<z\,\},\qquad &O_{2}=\{\,(x,y,z)\in \R^3 : x<z<y\,\}\,,\\
O_{3}=\{\,(x,y,z)\in \R^3 : y<x<z\,\},\qquad &O_{4}=\{\,(x,y,z)\in \R^3 : y<z<x\,\}\,,\\
O_{5}=\{\,(x,y,z)\in \R^3 : z<x<y\,\},\qquad &O_{6}=\{\,(x,y,z)\in \R^3 : z<y<x\,\}\,.
\end{align*}
The six open sets $O_{1},\dots,O_{6}$ are a partition of $\R^3\backslash \Hc$. On each of these open sets, the Jacobian of $\phi$ does not vanish. The set $\Hc$ is negligible with respect to the Lebesgue measure on $\R^3$, thus
\begin{align*}
\int_{\R^3}\!\Phi(x,y,z)\,d\nu_{\point \r}^{*3}(x,y,z)&=\sum_{i=1}^6\int_{O_{i} }\Phi(\phi(x,y,x))f(x)f(y)f(z)\,dx\,dy\,dz\\
&=\sum_{i=1}^6\int_{O_{i}}\Phi(\phi(x,y,x))\,g(x,y,z)\left|\mathrm{Jac}_{(x,y,z)}\phi\right|\,dx\,dy\,dz,
\end{align*}
where $g$ is the function defined on $\R^3\backslash \Hc$ by
\[\forall (x,y,z) \in \R^3\backslash \Hc\qquad g(x,y,z)=\frac{f(x)f(y)f(z)}{\left|8(x+y+z)(x-y)(y-z)(z-x)\right|}\,.\]

\noindent On each open set $O_{i}$, $i\in\{1,\dots,6\}$, the function $\phi$ is $\Ck{1}$ and its Jacobian does not vanish. In order to apply the global version of the inverse function theorem, we still have to prove that $\phi$ is one to one on each of these open sets. Let $(u,v,w) \in \R^3$ such that there exists $(x,y,z) \in \R^3 \backslash \Hc$ which verifies $(u,v,w)=\phi(x,y,z)$. We have then
\[x+y=u-z, \qquad x^2+y^2=v-z \qquad \mbox{et} \qquad x^4+y^4=w-z^4\,.\]
We search a polynomial equation satisfied by $z$. We have
\begin{align*}
\left( x+y \right) ^{2} \left( {x}^{2}+{y}^{2} \right) &={x}^{4}+{y}^{4}+2\,{x}^{2}{y}^{2}+2\,{x}^{3}y+2\,x{y}^{3}\\
&={x}^{4}+{y}^{4}+2\,xy \left(xy+{x}^{2}+{y}^{2}\right)\\
&={x}^{4}+{y}^{4}+\left((x+y)^2-(x^2+y^2)\right) \frac{(x+y)^2+x^2+y^2}{2}\,.
\end{align*}
Hence
\[\left( u-z \right) ^{2} \left( v-{z}^{2} \right) =w-{z}^{4}+ \left( 
 \left( u-z \right) ^{2}-v+{z}^{2} \right)\frac{(u-z)^2+v-z^2}{2}\,.\]
By developping, we get
\[4\,u{z}^{3}-4\,{u}^{2}{z}^{2}+2u\left( u^2-v \right) z+{u}^{2}v-\frac{u^4}{2}+\frac{v^2}{2}-w=0\,.\]
Since $(x,y,z) \notin \Hc$, we have $u=x+y+z \neq 0$ and thus $P_{(u,v,w)}(z)=0$ with
\[P_{(u,v,w)}(X)=X^{3}-uX^{2}+\frac{u^2-v}{2} X+\frac{uv}{4}-\frac{u^3}{8}+\frac{v^2}{8u}-\frac{w}{4u}\,.\]
Since $x$, $y$ and $z$ are exchangeable in the expression of $\phi$, we also have that
\[P_{(u,v,w)}(x)=P_{(u,v,w)}(y)=0\,.\]

\noindent Let $i\in\{1,\dots,6\}$. We have shown that, if $(u,v,w) \in \R^3$ is such that there exists $(x,y,z) \in O_i$ satisfying $(u,v,w)=\phi(x,y,z)$, then $x$, $y$ and $z$ are the zeros of the cubic polynomial $P_{(u,v,w)}$. As a consequence $\phi$ is one to one on $O_i$. Thus, by the global version of the inverse function theorem (see theorem 3.8.10~of~\cite{Schwartz2}), for any $i\in\{1,\dots,6\}$, the map $\phi$ is a $\Ck{1}$-diffeomorphism from $O_{i}^+$ to $\phi(O_{i}^+)$. We denote by $\phi_{i}^{-1}$ its inverse function.\medskip

\noindent Since $x$, $y$ and $z$ are exchangeable in the expression of $\phi$, we get that all the sets $\phi(O_i)$, $i\in\{1,\dots,6\}$, are equal to some set $\Uc$ and, for any $(u,v,w)\in \Uc$, the coordinates of $\phi_{i}^{-1}(u,v,w)$, $i\in\{1,\dots,6\}$, are the same up to a non-trivial permutation. As a consequence
\[\forall i\in\{1,\dots,6\} \qquad g \circ \phi_{i}^{-1}=g \circ \phi_{1}^{-1}\,.\]
Let us make the change of variables given by $\phi$ on each open set $O_{i}$ :
\[ \int_{\R^3}\Phi(x,y,z)\,d\nu_{\point \r}^{*3}(x,y,z)=\sum_{i=1}^6\int_{\phi(O_{i})}\Phi(u,v,w)\,g \circ \phi_{i}^{-1}(u,v,w)\,du\,dv\,dw\,.\]
The previous remarks about the symmetric structure of $\phi$ imply that
\[ \int_{\R^3}\Phi(x,y,z)\,d\nu_{\point \r}^{*3}(x,y,z)=\int_{\R^3}\Phi(u,v,w)\,6\,g \circ \phi_{1}^{-1}(u,v,w)\ind{\Uc}(u,v,w)\,du\,dv\,dw\,.\]
Hence $\nu_{\point \r}^{*3}$ admits the density $f_3=6\,g \circ \phi_{1}^{-1}\times\ind{\Uc}$ with respect to the Lebesgue measure on $\R^3$.\medskip

\noindent Next, for any $p \in [1,+\infty[$, we have
\[\int_{\R^3}f_3^p(u,v,w)\,du\,dv\,dw=6^p\int_{\Uc}\left(g \circ \phi_{1}^{-1}(u,v,w)\right)^p\,du\,dv\,dw\,.\]
Let us make the change of variables given by $\phi_{1}^{-1}$ :
\[\int_{\R^3}f_3^p(u,v,w)\,du\,dv\,dw=6^p\int_{O_{1}}g^p(x,y,z)\left|\mathrm{Jac}_{(x,y,z)}\phi\right|\,dx\,dy\,dz\,.\]
By symmetry, we write this integral
\begin{align*}
\int_{\R^3}f_3^p(u,v,w)\,du\,dv\,dw&=6^{p-1}\sum_{i=1}^6\int_{O_{i}}g^p(x,y,z)\left|\mathrm{Jac}_{(x,y,z)}\phi\right|\,dx\,dy\,dz\\
&=6^{p-1}\int_{\R^3}g^p(x,y,z)\left|\mathrm{Jac}_{(x,y,z)}\phi\right|\,dx\,dy\,dz\,.
\end{align*}
This is equal to
\[ \left(\frac{3}{4}\right)^{p-1}\int_{\R^3}\frac{f^p(x)f^p(y)f^p(z)}{\left|(x+y+z)(x-y)(y-z)(z-x)\right|^{p-1}}\,dx\,dy\,dz\,.\]
As a consequence, if this integral is finite for some $p \in  \, ]1,2]$, then $f_3 \in \Ll^p$. Thus proposition~A.6\footnote{Actually it is proposition 16 of to the ARXIV version of~\cite{CerfGorny}.}~of~\cite{CerfGorny} implies that, for $n$ large enough, $\widetilde{\nu}_{\point n,\r}$ has a density $g_{n}$ with respect to the Lebesgue measure on $\R^{3}$ such that, for any compact subset $K$ of $A_{I_{\point }}$, when $n \to +\infty$, uniformly over $(x,y,z)\in K$,
\[g_{n}(x,y,z)\sim \left(\frac{n}{2\pi}\right)^{3/2}\left(\mathrm{det}\, \mathrm{D}_{(x,y,z)}^{3}I_{\point }\right)^{1/2} e^{-nI_{\point }(x,y,z)}\,.\]
This ends the proof of the proposition.
\end{proof}
\bigskip

\noindent Let us prove that, if $f$ is bounded, then  there exists $p \in \,]1,2]$ such that
\[\int_{\R^3}\frac{f^{p}(x)f^{p}(y)f^{p}(z)}{\left|(x+y+z)(x-y)(y-z)(z-x)\right|^{p}}\,dx\,dy\,dz<+ \infty\,.\]
Young's inequality implies that, for any positive real numbers $a,b,c$ and $d$,
\[\frac{1}{abcd}\leq \frac{1}{2}\left(\frac{1}{(ab)^2}+\frac{1}{(cd)^2}\right)\leq \frac{1}{4}\left(\frac{1}{a^4}+\frac{1}{b^4}+\frac{1}{c^4}+\frac{1}{d^4}\right)\,.\]
By this inequality and by the symmetry of the integral in $x$, $y$ and $z$ we have
\[\int_{\R^3}\frac{f^{p}(x)f^{p}(y)f^{p}(z)}{\left|(x+y+z)(x-y)(y-z)(z-x)\right|^{p-1}}\,dx\,dy\,dz\leq \frac{1}{4}\left(I_1+3I_2\right)\,,\]
with
\[I_1=\int_{\R^3}\frac{f^{p}(x)f^{p}(y)f^{p}(z)}{\left|x+y+z\right|^{4(p-1)}}\,dx\,dy\,dz\,,\]
\[I_2=\int_{\R^3}\frac{f^{p}(x)f^{p}(y)f^{p}(z)}{\left|x-y\right|^{4(p-1)}}\,dx\,dy\,dz\,.\]
Making the change of variable $(x,y,z) \longmapsto(x+y+z,y,z)$, we get
\begin{align*}
I_1&=\int_{\R^3}\frac{f^{p}(u-v-w)f^{p}(v)f^{p}(w)}{\left|u\right|^{4(p-1)}}\,du\,dv\,dw\\
&=\int_{[-1,1]\times\R^2}\frac{f^{p}(u-v-w)f^{p}(v)f^{p}(w)}{\left|u\right|^{4(p-1)}}\,du\,dv\,dw\\
&\qquad\qquad\qquad\qquad\qquad\qquad+\int_{[-1,1]^c\times\R^2}\frac{f^{p}(u-v-w)f^{p}(v)f^{p}(w)}{\left|u\right|^{4(p-1)}}\,du\,dv\,dw\\
&\leq \|f\|_{\infty}^p \int_{[-1,1]\times\R^2}\frac{f^{p}(v)f^{p}(w)}{\left|u\right|^{4(p-1)}}\,du\,dv\,dw\\
&\qquad\qquad\qquad\qquad\qquad\qquad+\int_{[-1,1]^c\times\R^2}f^{p}(u-v-w)f^{p}(v)f^{p}(w)\,du\,dv\,dw\,.
\end{align*}
Fubini's theorem implies that
\[I_1\leq\|f\|_{\infty}^p \left(\int_{-1}^1\frac{du}{\left|u\right|^{4(p-1)}}\right)\left(\int_{\R}f^{p}(x)\,dx\right)^2+\left(\int_{\R}f^{p}(x)\,dx\right)^3\,.\]
We have
\[\int_{\R}f^{p}(x)\,dx=\int_{\R}f^{p-1}(x)f(x)\,dx\leq \|f\|_{\infty}^{p-1}\int_{\R}f(x)\,dx<+\infty\,.\]
Thus $I_1<+\infty$ as soon as $p<5/4$, since the function $u\longmapsto |u|^{4(1-p)}$ is then integrable on $[-1,1]$. We show similarly that $I_2<+\infty$ as soon as $p<5/4$. Hence
\[\forall p \in [1,5/4[ \qquad\int_{\R^3}\frac{f^{p}(x)f^{p}(y)f^{p}(z)}{\left|(x+y+z)(x-y)(y-z)(z-x)\right|^{p}}\,dx\,dy\,dz<+ \infty\,.\]

\subsection{Proof of theorem~\ref{MainTheo}}
\label{FluctuationsTheoremProof}

\noindent We prove in fact a more general fluctuation theorem than theorem~\ref{MainTheo}.\medskip

\noindent Let $\r$ be a symmetric probability measure on $\R$ and let $H$ be a real-valued function defined on $\Theta^{\!*}$ such that, for any $n\geq 1$,
\begin{multline*}
Z_{H,n}=\int_{\R^{n}}\exp\left(H\left(x_1+\cdots+x_n,\,x_1^2+\cdots+x_n^2,\,x_1^4+\cdots+x_n^4\right)\right)\\
\times\ind{\{x_{1}^{2}+\dots+x_{n}^{2}>0\}}\,\prod_{i=1}^{n}d\r(x_{i})<+\infty\,.
\end{multline*}

\noindent We consider $(X_{n}^{k})_{1\leq k \leq n}$ an infinite triangular array of real-valued random variables  such that, for all $n \geq 1$, $(X^{1}_{n},\dots,X^{n}_{n})$ has the law $\widetilde{\mu}_{H,n,\r}$, which is the distribution with density
\[\frac{1}{Z_{H,n}}\exp\left(H\left(x_1+\cdots+x_n,\,x_1^2+\cdots+x_n^2,\,x_1^4+\cdots+x_n^4\right)\right)\ind{\{x_{1}^{2}+\dots+x_{n}^{2}>0\}}\]
with respect to $\r^{\otimes n}$. We denote
\[S_{n}=X^{1}_{n}+\dots+X^{n}_{n}, \quad T_{n}=(X^{1}_{n})^{2}+\dots+(X^{n}_{n})^{2} \quad \mbox{and} \quad U_{n}=(X^{1}_{n})^{4}+\dots+(X^{n}_{n})^{4}.\]
We have the following general fluctuation theorem :

\begin{theo} Let $\r$ be a symmetric probability measure on $\R$ whose support contains at least five points and such that
\[ \exists w_0>0 \qquad  \int_{\R} e^{w_0 z^4}\,d\r(z)<+\infty\,.\]
We denote by $\s^2$ the variance of $\r$, by $\mu_4$ its the fourth moment, by $\mu_6$ its sixth moment and by $\mu_8$ its eighth moment. We assume~that
\[5\mu_{4}^{2}>2\s^{2}\mu_{6}\,.\]
Suppose that $H$ satisfies the hypothesis~\ref{hypHrho} and that $\r$ and $H$ fulfill either the hypothesis~\ref{hypHrho2} or the hypothesis~\ref{hypHrho2bis}. Then, under $\widetilde{\mu}_{H,n,\r}$, $(S_{n}/n,T_{n}/n,U_n/n)$ converges in probability towards $(0,\s^{2},\mu_4)$.\medskip

\noindent Moreover, if $\r$ has a density $f$ with respect to the Lebesgue measure on $\R$ such that, for some $p \in  \, ]1,2]$,
\[\qquad \quad\,\,\,\, \int_{\R^3}\frac{f^p(x)f^p(y)f^p(z)}{\left|(x+y+z)(x-y)(y-z)(z-x)\right|^{p-1}}\,dx\,dy\,dz<+\infty\,,\,\,\,\,\quad\qquad (*)\]
then, under $\widetilde{\mu}_{H,n,\r}$
\[\left(\frac{\mu_4^2}{\s^2}-\frac{2\mu_6}{5}\right)^{1/6}\frac{S_{n}}{\s^2n^{5/6}} \overset{\Lc}{\underset{n \to \infty}{\longrightarrow}} \left(\frac{81}{2}\right)^{1/6} \G\left(\frac{1}{6}\right)^{-1}\exp\left(-\frac{s^6}{18}\right)\,ds\,.\]
\label{TheoremModeleX4}
\end{theo}

\noindent We proved in the previous subsection that, if $f$ is bounded, then it satisfies $(*)$ for any $p\in \,]1,5/4[$. We have also proved in section~\ref{ConstructionInteraction}.\ref{candidate} that the function
\[H:(x,y,z) \in \Theta^{\!*} \longmapsto \frac{x^{2}}{2y}+\frac{1}{12}\frac{zx^{4}y^{5}}{y^{9}+x^{10}+zx^4y^4}\]
satisfies the hypothesis~\ref{hypHrho} and~\ref{hypHrho2bis}. Hence theorem~\ref{MainTheo} is a consequence of this theorem.\medskip

\noindent Considering our article~\cite{Gorny3}, we could prove this fluctuation theorem for $\r$ having an absolutely continuous component (and not necessarily having a density which satisfies $(*)$) or, more generally, for $\r$ satisfying the Cram\'er condition
\begin{equation}
\forall \a>0 \qquad \sup_{\|(s,t,u)\|\geq \a}\left|\int_{\R}e^{isz+itz^2+iuz^4}\,d\r(z)\right|<1.\tag{$C$}
\end{equation}
However the proof would be much more technical.\medskip

\noindent{\bf Proof of proposition~\ref{TheoremModeleX4}.} We denote by $\theta_{\point n,\r}$ the law of $(S_{n}/n,T_{n}/n,U_n/n)$ under $\widetilde{\mu}_{\point n,\r}$. Let $U$ be an open neighbourhood of $(0,\s^{2},\mu_4)$ in $\R^{3}$. Suppose that $\r$ and $H$ satisfy the hypothesis~\ref{hypHrho} and also either the hypothesis~\ref{hypHrho2} or the hypothesis~\ref{hypHrho2bis}. Then proposition~\ref{TypeVaradhan-x4} implies that
\begin{align*}
\limsupn \frac{1}{n}\ln \theta_{\point n,\r}(U^{c})&=\limsupn \frac{1}{n}\ln \int_{\Theta^{\!*}\cap U^{c}}e^{nF(x,y)+nR_n(x,y,z)}\,d\widetilde{\nu}_{\point n,\r}(x,y,z)\\
&\qquad -\liminfn \frac{1}{n}\ln \int_{\Theta^{\!*}}e^{nF(x,y)+nR_n(x,y,z)}\,d\widetilde{\nu}_{\point n,\r}(x,y,z)<0.
\end{align*}
Hence there exist $\eps >0$ and $n_{0}\in \N$ such that for any $n>n_{0}$, 
\[ \theta_{\point n,\r}(U^{c}) \leq e^{-n\eps} \underset{n \to \infty}{\longrightarrow} 0\,.\]
Thus, for each open neighbourhood $U$ of $(0,\s^{2},\mu_4)$,
\[\lim_{n \to +\infty}\widetilde{\mu}_{\point n,\r}\left(\left(\frac{S_{n}}{n},\frac{T_{n}}{n},\frac{U_{n}}{n}\right)\in U^{c}\right)=0\,.\]
This means that, under $\widetilde{\mu}_{\point n,\r}$, $(S_{n}/n,T_{n}/n,U_n/n)$ converges in probability to $(0,\s^{2},\mu_4)$.\medskip

\noindent Next, in section~\ref{ConstructionInteraction}.\ref{FirstInvestigations}, we proved that, in the neighbourhood of $(0,\s^{2},\mu_{4})$,
\[I_{\point }(x,y,z)-F(x,y)-R(x,y,z) \sim \frac{(5\mu_{4}^{2}-2\s^{2}\mu_{6})x^6}{90\s^{14}}+q(y-\s^{2},z-\mu_{4})\,.\]
Thus, there exists $\d>0$ such that, for any $(x,y,z) \in \Bl_{\d}$,
\[(I_{\point}-F-R)(x,y,z) \geq \frac{(5\mu_{4}^{2}-2\s^{2}\mu_{6})x^6}{180\s^{14}}+\frac{1}{2}q(y-\s^{2},z-\mu_{4})\,,\]
where $\Bl_{\d}$ denotes the open ball of radius $\d$ centered at $(0,\s^{2},\mu_4)$. We have $R_n\leq R$ for any $n\geq 1$, thus, for any $(x,y,z) \in \Bl_{\d}$ and $n\geq 1$,
\begin{equation}
(I_{\point}-F-R_n)(x,y,z) \geq \frac{(5\mu_{4}^{2}-2\s^{2}\mu_{6})x^6}{180\s^{14}}+\frac{1}{2}q(y-\s^{2},z-\mu_{4}).
\label{**}
\end{equation}
We can reduce $\d$, in order to have $\Bl_{\d}\subset K$ where $K$ is a compact subset of $A_{I_{\point}}$ so that $\Bl_{\d} \subset A_{I_{\point}} \subset  \Theta^{\!*}$.\medskip

\noindent Let $n \geq 1$ and let $f :\R \longrightarrow \R$ be a bounded continuous function. We have
\begin{align*}
\E_{\tilde{\mu}_{n,\r}}\left(f\left(\frac{S_{n}}{n^{5/6}}\right)\right)&=\frac{1}{Z_{H,n}}\int_{\Theta^{\!*}}f(xn^{1/6})\,e^{nF(x,y)+nR_n(x,y,z)}\,d\widetilde{\nu}_{\point n,\r}(x,y,z)\\
&=\frac{A_{n}+B_{n}}{Z_{H,n}}\,,
\end{align*}
with
\[A_{n}=\int_{\Bl_{\d}}f(xn^{1/6})\,e^{nF(x,y)+nR_n(x,y,z)}\,d\widetilde{\nu}_{\point n,\r}(x,y,z)\,,\]
\[B_{n}=\int_{\Theta^{\!*}\cap\Bl^{c}_{\d}}f(xn^{1/6})\,e^{nF(x,y)+nR_n(x,y,z)}\,d\widetilde{\nu}_{\point n,\r}(x,y,z)\,.\]

\noindent Suppose in addition that $\r$ has a density $f$ with respect to the Lebesgue measure on $\R$ such that, for some $p \in  \, ]1,2]$, the function
\[(x,y,z)\longmapsto \frac{f^p(x)f^p(y)f^p(z)}{\left|(x+y+z)(x-y)(y-z)(z-x)\right|^{p-1}}\]
is integrable. Then proposition~\ref{densitef3} implies that, for $n$ large enough, $\widetilde{\nu}_{\point n,\r}$ has a density $g_{n}$ with respect to the Lebesgue measure on $\R^{3}$. Let us introduce the factor $e^{-nI_{\point}(x,y,z)}$ in the expression of $A_{n}$ :\[A_{n}=n^{3/2}\int_{\Bl_{\d}}f(xn^{1/6})\,e^{-nG_n(x,y,z)}H_{n}(x,y,z)\,dx\,dy\,dz\,,\]
where $G_n=I_{\point}-F-R_n$ and
$H_{n}:(x,y,z) \longmapsto e^{nI_{\point}(x,y,z)}g_{n}(x,y,z)/n^{3/2}$.
We define
\[\Bl_{\d,n}=\{\,(x,y,z) \in \R^{3} : x^{2}/n^{1/3}+y^{2}/n+z^{2}/n \leq \d^{2}\,\}\,.\]
Let us make the change of variables given by
\[(x,y,z) \longmapsto (xn^{-1/6},yn^{-1/2}+\s^{2},zn^{-1/2}+\mu_4)\,,\] with Jacobian $n^{-7/6}$ :
\begin{multline*}
A_{n}=n^{1/3}\int_{\Bl_{\d,n}}f(x)\exp\left(-nG_n\left(\frac{x}{n^{1/6}},\frac{y}{\sqrt{n}}+\s^{2},\frac{z}{\sqrt{n}}+\mu_4\right)\right)\\
\times H_{n}\left(\frac{x}{n^{1/6}},\frac{y}{\sqrt{n}}+\s^{2},\frac{z}{\sqrt{n}}+\mu_4\right)\,dx\,dy\,dz\,.
\end{multline*}
We check now that we can apply the dominated convergence theorem to this integral. The uniform expansion of $g_n$ given by proposition~\ref{densitef3} means that for any $\a>0$, there exists $n_{0} \in \N$ such that, for any $n\geq n_0$,
\[\forall (x,y,z) \in K \qquad \left|H_{n}(x,y,z)\,(2\pi)^{3/2} \left(\mathrm{det}\, \mathrm{D}_{(x,y,z)}^{2}I_{\point }\right)^{-1/2} -1\right|\leq \a\,.\]
If $(x,y,z) \in \Bl_{\d,n}$, then $(x_{n},y_{n},z_n)=(xn^{-1/6},yn^{-1/2}+\s^{2},zn^{-1/2}+\mu_4) \in \Bl_{\d}$, which is included in $K$. Thus for all $n\geq n_{0}$ and $(x,y,z) \in \Bl_{\d,n}$,
\[\left|H_{n}(x_{n},y_{n},z_n)\,(2\pi)^{3/2} \left(\mathrm{det}\, \mathrm{D}_{(x_{n},y_{n},z_n)}^{2}I_{\point }\right)^{-1/2} -1\right|\leq \a\,.\]
Moreover $(x_{n},y_{n},z_n) \to (0,\s^{2},\mu_4)$ thus, by continuity,
\[\left(\mathrm{D}_{(x_{n},y_{n},z_n)}^{2}I_{\point }\right)^{-1/2} \underset{n\to +\infty}{\longrightarrow} \left(\mathrm{D}_{(0,\s^{2})}^{2}I_{\point } \right)^{-1/2} = \left(\mathrm{D}_{(0,0)}^{2}\L_{\point }\right)^{1/2}\,,\]
whose determinant is $\sqrt{\s^{2}a}$, with $a=(\mu_4 - \s^4)(\mu_8-\mu_4^2)-(\mu_6-\s^2\mu_4)^2$. It is positive according to lemma~\ref{Support5}. Therefore
\[\ind{\Bl_{\d,n}}(x,y,z)H_{n}\left(x_{n},y_{n},z_n\right)\underset{n\to +\infty}{\longrightarrow} \left((2\pi)^{3} \s^{2}a\right)^{-1/2}\,.\]
Next, proposition~\ref{MinEq(I*-Fn)} implies that, for any $(x,y,z) \in \R^3$,
\[\exp\left(-nG_n\left(x_{n},y_{n},z_n\right)\right)\underset{n \to + \infty}{\longrightarrow}\exp\left(-q(y,z)-\frac{5\mu_{4}^{2}-2\s^{2}\mu_{6}}{90\s^{14}}x^{6}\right)\,.\]
Let us check that the integrand is dominated by an integrable function, which is independent of $n$. The function
\[(x,y,z) \longmapsto \left(\mathrm{D}_{(x,y,z)}^{2}I_{\point}\right)^{-1/2}\]
is bounded on $\Bl_{\d}$ by some $M_{\d}>0$. The uniform expansion of $g_{n}$ implies that for all $(x,y,z) \in \Bl_{\d}$, $H_{n}(x,y,z)\leq C_{\d}$ for some constant $C_{\d}>0$. Finally, the inequality~(\ref{**}) above yields that, for any $(x,y,z)\in \R^3$,
\begin{multline*}
\ind{\Bl_{\d,n}}(x,y,z)f(x)\exp\left(-nG_n\left(x_{n},y_{n},z_n\right)\right)\, H_{n}\left(x_{n},y_{n},z_n\right)\\
\leq \|f\|_{\infty}C_{\d}\exp\left(-\frac{1}{2}q(y,z)-\frac{5\mu_{4}^{2}-2\s^{2}\mu_{6}}{180\s^{14}}x^{6}\right)\,.
\end{multline*}
The right term defines an integrable function on $\R^{3}$, thus it follows from the dominated convergence theorem that
\[A_{n}\underset{+\infty}{\sim}n^{1/3}\int_{\R^{3}}\frac{f(x)}{\sqrt{2 \pi \s^{2}}(2 \pi\sqrt{a})}\exp\left(-q(y,z)-\frac{5\mu_{4}^{2}-2\s^{2}\mu_{6}}{90\s^{14}}x^{6}\right)dx\,dy\,dz\,.\]
Fubini's theorem implies that, for some constant $k>0$,
\[A_{n}\underset{+\infty}{\sim}k n^{1/3}\int_{\R}f(x)\exp\left(-\frac{5\mu_{4}^{2}-2\s^{2}\mu_{6}}{90\s^{14}}x^{6}\right)\,dx\,.\]

\noindent Let us focus now on $B_{n}$. Proposition~\ref{TypeVaradhan-x4} implies that there exists $\eps>0$ such that, for $n$ large enough,
$B_{n} \leq \|f\|_{\infty}e^{-n\eps}$ and thus $B_{n}=o(n^{1/3})$. Therefore
\[A_{n}+B_{n}\underset{+\infty}{\sim}k n^{1/3}\int_{\R}f(x)\exp\left(-\frac{5\mu_{4}^{2}-2\s^{2}\mu_{6}}{90\s^{14}}x^{6}\right)\,dx\,.\]
Applying this to $f=1$, we get
\[Z_{n}\!\underset{+\infty}{\sim}\!kn^{1/3}\!\int_{\R}\!\exp\left(\!-\frac{5\mu_{4}^{2}-2\s^{2}\mu_{6}}{90\s^{14}}x^{6}\!\right)\,dx=3k n^{1/3}\left(\!\frac{5\mu_{4}^{2}-2\s^{2}\mu_{6}}{90\s^{14}}\!\right)^{-1/6}\!\Gamma\left(\frac{1}{6}\right).\]
Finally
\[\E_{\tilde{\mu}_{n,\r}}\left(f\left(\frac{S_{n}}{n^{3/4}}\right)\right)\underset{+\infty}{\sim}
\frac{\displaystyle{\int_{\R}f(x)\,\exp\left(-\frac{5\mu_{4}^{2}-2\s^{2}\mu_{6}}{90\s^{14}}x^{6}\right)\,dx}}{\displaystyle{3 \left(\frac{5\mu_{4}^{2}-2\s^{2}\mu_{6}}{90\s^{14}}\right)^{-1/6} \Gamma\left(\frac{1}{6}\right) }}\,.\]
The ultimate change of variables $y=(5\mu_{4}^{2}-2\s^{2})^{1/6}x/(5\s^{14})^{1/6}$ gives us theorem~\ref{TheoremModeleX4}. \qed

\section{Fluctuations of order $n^{1-1/2k}$?}
\label{OrdreSuivant}

\noindent Let $k\geq 4$. We denote by $I_{\point k}$ the Cram\'er transform of $(Z,Z^2,Z^4,\dots,Z^{2k-2})$, where $Z$ is a random variable with distribution $\r$. We would like to find a large class of probability measures $\r$ on $\R$ such that :\smallskip

\noindent $\star$ There exists an interacting function $H_k$ from $\R^k$ to $\R$ such that, for any $n \geq 1$, the function
\[G_{n,k} : (y_1,y_2,\dots,y_{2k-2})\longmapsto I_{\point k}(y_1,y_2,\dots,y_{2k-2})-\frac{1}{n}H_k(ny_1,ny_2,\dots,ny_{2k-2})\]
admits a unique minimum at $(0,\s^2,\mu_4,\dots,\mu_{2k-2})$, where $\s^2,\mu_4,\dots,\mu_{2k-2}$ are the successive moments of $\r$.\smallskip

\noindent $\star$ For any $n \geq 1$, we denote by $Z_{n,k}$ the integral
\[\int_{\R^n}\exp\left(H_k\left(x_1+\cdots+x_n,\,x_1^2+\cdots+x_n^2,\,\dots,\,x_1^{2k}+\cdots+x_n^{2k}\right)\right)\,\prod_{i=1}^n\,d\r(x_i)\]
and we suppose it is finite.\smallskip

\noindent $\star$ There exist $A_k>0$ and a function $q_k$ from $\R^{k-1}$ to $\R$ satisfying
\[\int_{\R}e^{-q_k(y_2,\dots,y_{2k-2})/2}\,dy_2 \,\dots\,dy_{2k-2}<+\infty\]
such that, in a neighbourhood of $(0,\s^2,\mu_4,\dots,\mu_{2k-2})$, 
\[G_{n,k}(y_1,y_2,\dots,y_{2k-2})\sim A_k y_1^{2k} + q_k(y_2-\s^2,\dots,y_{2k-2}-\mu_{2k-2})\,.\]

\noindent In this case, we consider $(X^{n}_{k})_{1\leq k \leq n}$ an infinite triangular array of real-valued random variables such that, for all $n \geq 1$, $(X^{1}_{n},\dots,X^{n}_{n})$ has the distribution 
\[\frac{1}{Z_{n,k}}\exp\left(H_k\left(x_1+\cdots+x_n,\,x_1^2+\cdots+x_n^2,\,\dots,\,x_1^{2k}+\cdots+x_n^{2k}\right)\right)\,\prod_{i=1}^n\,d\r(x_i)\,.\]
We denote $S_n=X_1^n+\dots+X_n^n$ for any $n\geq 1$. By using arguments as in the last sections, we could prove the following : 
\[\frac{S_n}{n^{1-1/2k}} \overset{\Lc}{\underset {n \to \infty}{\longrightarrow}}   \left(\int_{\R}\exp(-A_k y^{2k})\,dy\right)^{-1}\,\exp(-A_k x^{2k})\,dx\,.\]

\noindent Unfortunately the proof of such a result does not seem to be possible with the techniques we employed in this paper, for several reasons : \medskip

\noindent $\star$ In order to obtain the expansion of $G_{n,k}$, in the case $k=2$ or $3$, we made very long and tedious computations. Of course we could repeat these computations for $k=4$, then $k=5$, ... But it would be very complicated and this is not reasonable if we do not find a simple way to determine the variable $A_k$ for any $k\geq 4$. Moreover we have not understood why, for $k=2,3$, the terms \og we do not want \fg{} in the expansion of $G_{n,k}$ vanish.\medskip

\noindent $\star$ For $k=3$, there are probability measures such that $A_k$ is negative. In the same way, there may exist $k_0\geq 4$ such that $A_{k_{0}}<0$ for any probability measure. In this case, $G_{n,k_{0}}$ could not admit a minimum at $(0,\s^2,\mu_4,\dots,\mu_{2k-2})$ and we should find new criteria on $H_{k_0}$ to solve this problem.\medskip

\noindent $\star$ With the \og natural \fg{} interacting function in the case $k=3$, we have not managed to prove that $G_{n,3}$ has a unique minimum at $(0,\s^2,\mu_4)$ (while our simulations tend to conjecture this is true). We had to force the interacting function to have the desired behaviour by finding some suitable criteria. Moreover the candidate we propose for $H$ is rather complicated. We also failed to make convincing computer simulations with our modified model (although it is amenable to mathematical analysis) : the convergence is too slow because\[n(R-R_n)\left(\frac{x_1+\dots+x_n}{n},\frac{x_1^2+\dots+x^2_n}{n},\cdots,\frac{x_1^{2k-2}+\dots+x_n^{2k-2}}{n}\right)\]
becomes negligible only for very large $n$.

\bibliographystyle{plain}
\bibliography{biblio}

\begin{thebibliography}{1}

\bibitem{CerfGorny}
Rapha\"el Cerf and Matthias Gorny.
\newblock A {C}urie-{W}eiss model of self-organized criticality.
\newblock {\em preprint}, 2013.

\bibitem{DZ}
Amir Dembo and Ofer Zeitouni.
\newblock {\em Large deviations techniques and applications}, volume~38 of {\em
  Stochastic Modelling and Applied Probability}.
\newblock Springer-Verlag, 2010.

\bibitem{EN}
Richard~S. Ellis and Charles~M. Newman.
\newblock Limit theorems for sums of dependent random variables occurring in
  statistical mechanics.
\newblock {\em Z. Wahrsch. Verw. Gebiete}, 44(2):117--139, 1978.

\bibitem{Gorny3}
Matthias Gorny.
\newblock The {C}ram\'er {C}ondition for the {C}urie-{W}eiss model of {SOC}.
\newblock {\em preprint}, 2013.

\bibitem{RR}
Gareth~O. Roberts and Jeffrey~S. Rosenthal.
\newblock Harris recurrence of {M}etropolis-within-{G}ibbs and
  trans-dimensional {M}arkov chains.
\newblock {\em Ann. Appl. Probab.}, 16(4):2123--2139, 2006.

\bibitem{Rockafellar}
R.~Tyrrell Rockafellar.
\newblock {\em Convex analysis}.
\newblock Princeton Mathematical Series, No. 28. Princeton University Press,
  1970.

\bibitem{Schwartz2}
Laurent Schwartz.
\newblock {\em Analyse II : Calcul diff{\'e}rentiel et {\'e}quations
  diff{\'e}rentielles.}
\newblock Collection Enseignement des sciences. Hermann, 1992.

\end{thebibliography}
\addcontentsline{toc}{section}{References}
\markboth{\uppercase{References}}{\uppercase{References}}

%\tableofcontents

\end{document}